\pdfoutput=1
\documentclass[11pt,a4paper,reqno]{amsart}

\usepackage{amssymb}
\usepackage{amsmath}
\usepackage{amsthm}
\usepackage{thmtools}
\usepackage{mathtools}
\usepackage{mathabx}
\usepackage{mathrsfs}
\usepackage{soul}
\usepackage{bm}
\numberwithin{equation}{section}
\usepackage{framed}
\usepackage[bottom]{footmisc}
\usepackage[utf8]{inputenc}
\usepackage[colorlinks=false,hidelinks]{hyperref}
\usepackage[nameinlink,capitalize]{cleveref}
\usepackage{enumitem}
\usepackage{graphicx}
\usepackage[dvipsnames]{xcolor}
\usepackage{relsize}
\usepackage{exscale}
\usepackage{floatrow}
\usepackage{pdflscape}
\usepackage{afterpage}
\usepackage[justification=centering]{caption}
\usepackage{rotating}
\usepackage{pdflscape}
\usepackage{tabulary}
\usepackage{caption}
\usepackage{subcaption}

\theoremstyle{plain}

\newtheorem{theorem}[subsection]{Theorem}
\newtheorem{proposition}[subsection]{Proposition}
\newtheorem{lemma}[subsection]{Lemma}
\newtheorem{corollary}[subsection]{Corollary}

\theoremstyle{definition}
\newtheorem{definition}[theorem]{Definition}
\newtheorem{example}[theorem]{Example}

\makeatletter
\newcommand{\vast}{\bBigg@{4}}
\newcommand{\Vast}{\bBigg@{5}}
\makeatother

\makeatletter
\newcommand*\bigcdot{\mathpalette\bigcdot@{.5}}
\newcommand*\bigcdot@[2]{\mathbin{\vcenter{\hbox{\scalebox{#2}{$\m@th#1\bullet$}}}}}
\makeatother

\newcommand*{\medcup}{\mathbin{\scalebox{1.5}{\ensuremath{\cup}}}}
\newcommand*{\medcap}{\mathbin{\scalebox{1.5}{\ensuremath{\cap}}}}

\usepackage[square,numbers]{natbib}
\usepackage{amsaddr}
\usepackage{etoolbox}

\makeatletter
\patchcmd{\@maketitle}
{\ifx\@empty\@dedicatory}
{\ifx\@empty\@date \else {\vskip3ex \centering\@date\par\vskip1ex}\fi
	\ifx\@empty\@dedicatory}
{}{}
\patchcmd{\@maketitle}
{\ifx\@empty\@date\else \@footnotetext{\@setdate}\fi}
{}{}{}
\makeatother

\renewcommand{\leq}{\leqslant}
\renewcommand{\geq}{\geqslant}

\def\vs{\vspace{11pt}}
\def\ni{\noindent}
\def\emph#1{{\it #1}}
\def\textbf#1{{\bf #1}}

\usepackage{subcaption}
\begin{document}
	
\title{Asymptotic stochastic comparison of random processes}

\author{Sugata Ghosh and Asok K. Nanda}
\address{Department of Mathematics and Statistics\\
	Indian Institute of Science Education and Research Kolkata\\
	Mohanpur 741246, India}
\email{sg18rs017@iiserkol.ac.in}
\email{asok@iiserkol.ac.in}


\allowdisplaybreaks
\raggedbottom

\begin{abstract}
	Several methods are available in the literature to stochastically compare random variables and random vectors. We introduce the notion of asymptotic stochastic order for random processes and define four such orders. Various properties and interrelations of the orders are discussed. Sufficient conditions for these orders to hold for certain stochastic processes, evolving from some statistical entities of interest, are derived.
\end{abstract}

\maketitle
\markleft{\uppercase{Asymptotic stochastic order}}
\markright{\uppercase{Sugata Ghosh and Asok K. Nanda}}

\section{Introduction}\label{1-section-introduction}

One of the chief aims of probability and statistics is to develop methods to compare various random entities, which is a practical problem in several areas of studies. In the classical theory of comparison of random variables, stochastic orders provide useful measuring sticks to capture the notion of one random variable being larger or smaller than another, in some appropriate sense. Many such orders are defined and extensively studied in the literature, with areas of applications ranging from probability and statistics (see, for example, \citet{CC_2006}, \citet{MS_2006}, \citet{FK_2013}, \citet{BHSS_2002}, \citet{B_1991}, \citet{H_2017}) to actuarial science, risk management and economics (\citet{KO_1999}, \citet{BM_2006}, \citet{BB_2013}, \citet{W_2017}, \citet{LLM_2018}), operations research (\citet{FMR_2011}), wireless communications (\citet{TRZ_2011}) and other related fields. For a detailed exposition of various stochastic orders and their properties, we refer to \citet{MS_2002} and \citet{SS_2007}.\vs

While there is a rich literature on stochastic ordering of random variables and random vectors, stochastic comparison of random processes has enjoyed lesser attention due to the abstractness of the problem (see \citet{PP_1973}, \citet[chapter $4$]{S_2001}, \citet[chapter $5$]{MS_2002} and \citet{MS_2013}). In this work, we focus on the problem of stochastic comparison of general stochastic processes in an asymptotic sense. To be precise, let $\left\{X_t : t \in T\right\}$ and $\left\{Y_t : t \in T\right\}$ be two stochastic processes, where the common index set $T \subseteq \mathbb{R}$ is unbounded above. We are interested in stochastic comparison between $X_t$ and $Y_t$ from a limiting perspective, as $t \to \infty$.\vs

In this paper, we have defined and analyzed four asymptotic stochastic orders, namely asymptotic usual stochastic order, asymptotic stochastic precedence order, $L_1$-asymptotic usual stochastic order and $\mathcal{W}_2$-asymptotic usual stochastic order. The first two orders are the most straightforward approaches, based directly on the usual stochastic order and stochastic precedence order for random variables. We call these orders {\it semi-quantitative} in nature, due to their abandonment of certain information about the random variables or their respective cumulative distribution functions (cdf), which we have explained in \cref{1-section-motivation-quantitative}. The last two orders, which also are based on usual stochastic order, are more {\it quantitative} in nature and overcome certain issues (discussed in \cref{1-section-motivation-quantitative}) with the first two orders.\vs

The notion of distortion function, which is a simple mathematical tool for transforming one cdf to another, is defined in \citet{D_1994} as follows.
\begin{definition}
	Any function $\phi:[0,1] \to [0,1]$ that is nondecreasing with $\phi(0)=0$ and $\phi(1)=1$ is called a {\it distortion function} or {\it probability transformation function}. \hfill $\blacksquare$
\end{definition}

The concept is applied in several areas including reliability theory and stochastic ordering (see \citet{KS_2010}, \citet{SS_2011}, \citet{N_2013}), insurance pricing and financial risk managements (see \citet{W_1995, W_1996}), expected utility theory (see \citet{WY_1998}) etc. Now, let $X$ and $Y$ be two continuous random variables with respective cdfs $F_X$ and $F_Y$. Also, let $P_X$ and $P_Y$ denote the probability measures generated by $F_X$ and $F_Y$, respectively.
\begin{definition}\label{1-definition-usual-stochastic-order}
	$X$ is said to be less than $Y$ in {\it usual stochastic order} (denoted by $X \leq_{\textnormal{st}} Y$) if $F_X(x) \geq F_Y(x)$, for every $x \in \mathbb{R}$. Also, we say that $X$ is equal to $Y$ in usual stochastic order (denoted by $X =_{st} Y$) if $X \leq_{\textnormal{st}} Y$ and $Y \leq_{\textnormal{st}} X$. \hfill $\blacksquare$
\end{definition}

It can be easily verified that the usual stochastic order is a partial order. However, it lacks the {\it connex} property, i.e. given two random variables it may happen that none of $X \leq_{\textnormal{st}} Y$ and $Y \leq_{\textnormal{st}} X$ hold. Also, it does not capture any possible dependence structure between $X$ and $Y$ since it involves only the respective marginal distributions. \citet{AKS_2002} introduced the following stochastic order which serves as an alternative to the usual stochastic order.

\begin{definition}\label{1-definition-stochastic-precedence-order}
	$X$ is said to be less than $Y$ in {\it stochastic precedence order} (denoted by $X \leq_{\textnormal{sp}} Y$) if $P(X \leq Y) \geq 1/2$. \hfill $\blacksquare$
\end{definition}
$X$ and $Y$ are said to be equal in stochastic precedence order (denoted by $X =_{sp} Y$) if $X \leq_{\textnormal{sp}} Y$ and $Y \leq_{\textnormal{sp}} X$. The quantity $P(X \leq Y)$ serves as a simple measure of dominance of $Y$ over $X$. The analysis of this quantity was initiated by \citet{B_1956} and has great importance in stress-strength analysis (see \citet{K_2003}). The stochastic precedence order has the {\it connex} property and also captures the dependence structure between the two random variables under consideration. However, stochastic precedence order lacks the {\it transitivity property} and hence is not a partial order. This property becomes crucial when more than two random variables are involved in the problem of stochastic comparison. When $X$ and $Y$ are independent, it can be easily verified that $X \leq_{\textnormal{st}} Y$ implies $X \leq_{\textnormal{sp}} Y$.\vs

We emphasize on the fact that $F_X(x) \geq F_Y(x)$ must hold for every real number $x$ for the usual stochastic order to hold, which makes it a very strong condition. On the flip side, as observed by \citet{AKS_2002}, given two random variables $X$ and $Y$, the strength of the definition of usual stochastic order often ends up concluding that none of $X$ and $Y$ dominates the other, even if $F_X \geq F_Y$ holds throughout $\mathbb{R}$, except for an arbitrarily small interval.\vs

Also, it often happens, for example in comparison of the largest order statistics from two groups of independent and identically distributed (iid) random variables, that the distribution functions $F_X$ and $F_Y$ cross each other. However, the probability mass (according to both $P_X$ and $P_Y$) on the region where $F_X$ dominates $F_Y$, is negligible compared to the same on the region where $F_Y$ dominates $F_X$. To motivate the idea, let us consider the following situation. Let $F_X$ and $F_Y$ be the cdfs of standard normal distribution (denoted by $N(0,1)$) and standard Cauchy distribution (denoted by $C(0,1)$), respectively. Clearly, the right tail of $F_Y$ is heavier than that of $F_X$ in the sense that there exists $c \in \mathbb{R}$ such that $F_X(x)>F_Y(x)$, whenever $x>c$. Also, none of the two distributions dominate the other in the sense of usual stochastic order (see \cref{1-fig-N01C01-distribution-functions-ggb}). Now, if one takes two samples (of same size) of random observations, the first from $F_X$ and the second from $F_Y$, then one may expect that the largest observation from the former sample is smaller than that from the latter. However, the plots of the distribution functions of $X_{n:n}$ and $Y_{n:n}$ (with $n=10$) leaves a strong indication towards the dominance of $Y_{n:n}$ over $X_{n:n}$ (see \cref{1-fig-N01C01-largest-order-statistics-distribution-functions-ggb}).\vs
\begin{figure}
	\centering
	\begin{subfigure}{0.45\textwidth}
		\centering
		\includegraphics[width=\linewidth]{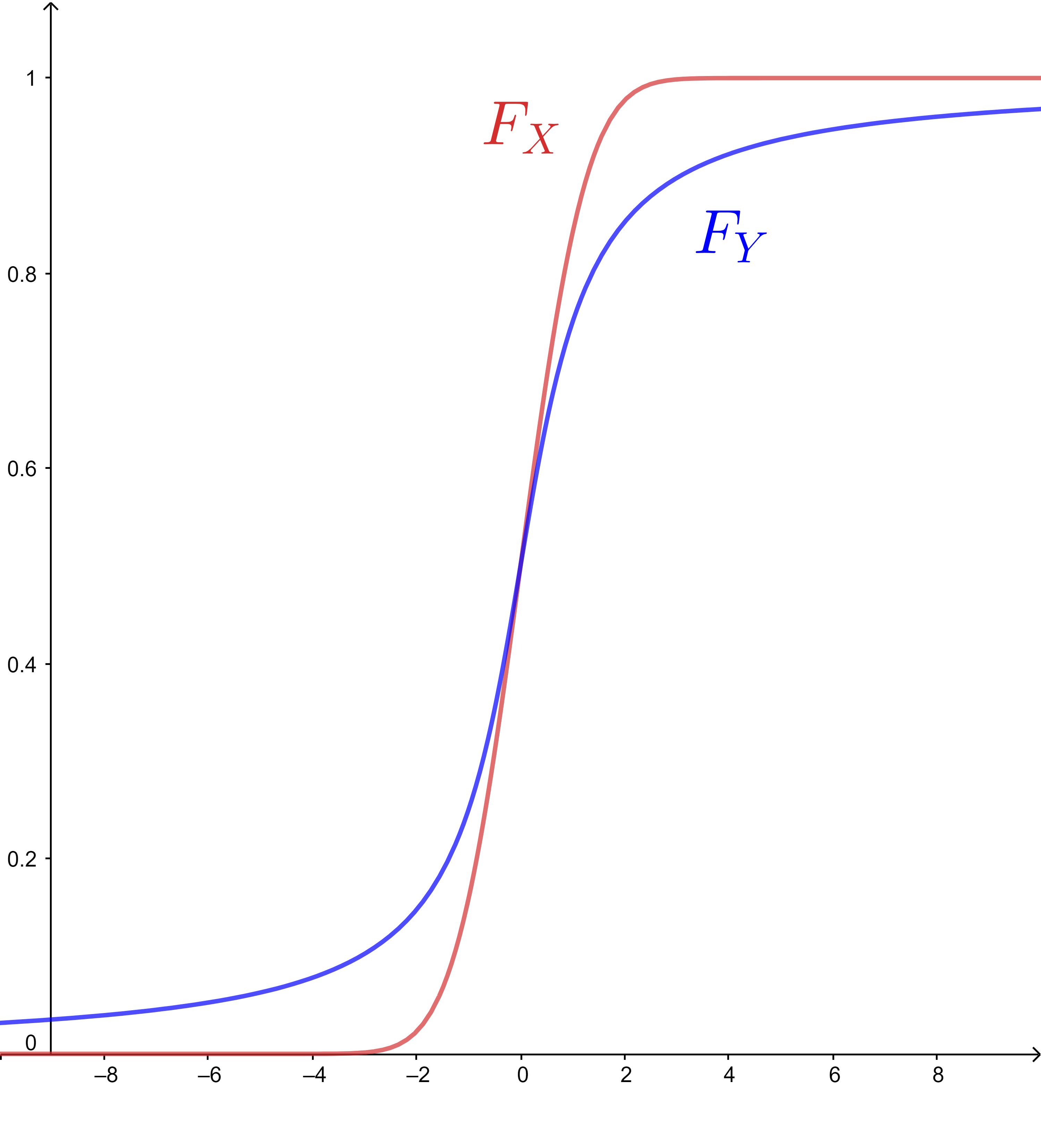}
		\caption{\centering Cumulative distribution functions of $N(0,1)$ and $C(0,1)$}
		\label{1-fig-N01C01-distribution-functions-ggb}
	\end{subfigure}
	\hfill
	\begin{subfigure}{0.45\textwidth}
		\centering
		\includegraphics[width=\linewidth]{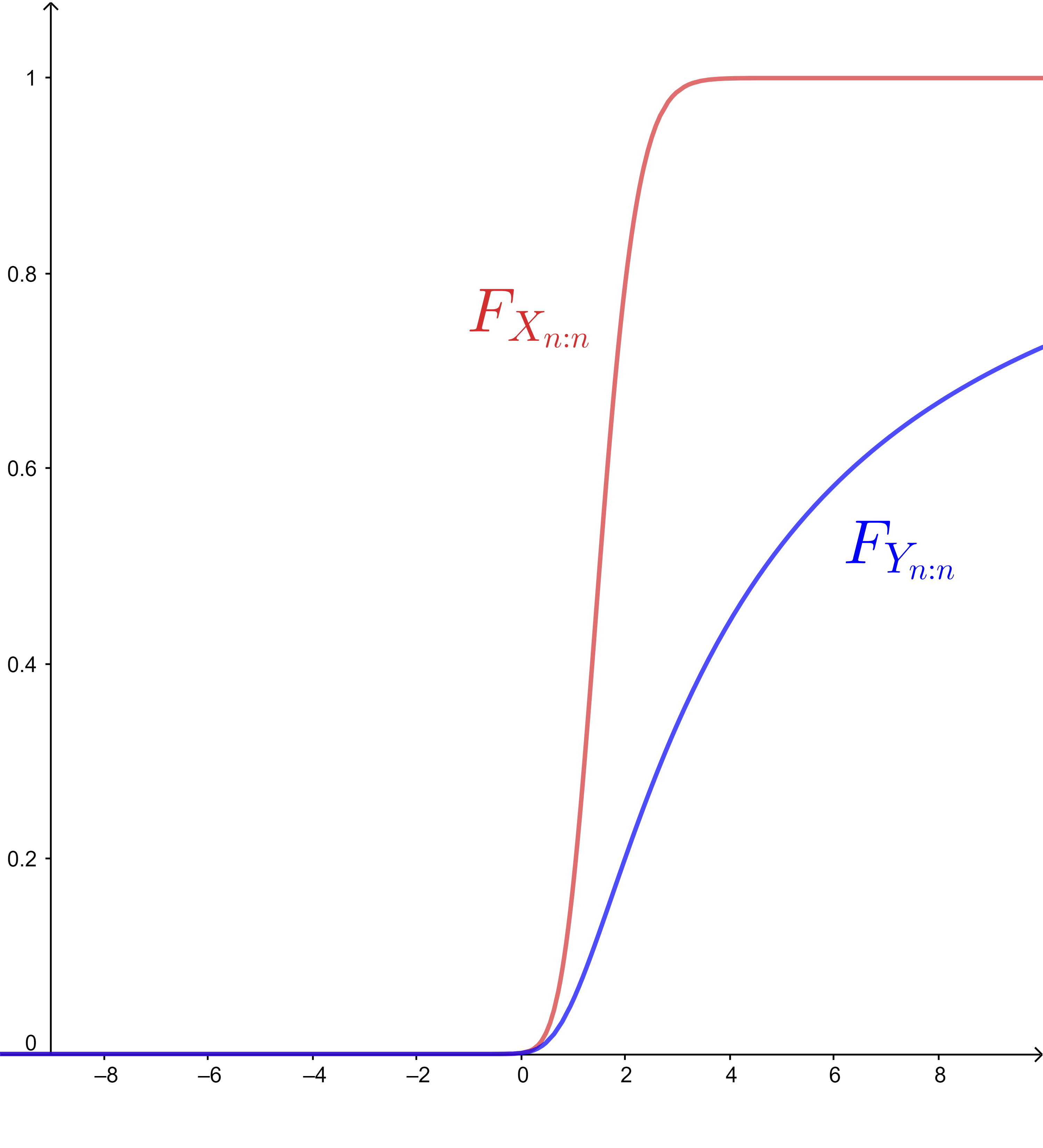}
		\caption{\centering CDFs of the largest order statistics from $N(0,1)$ and $C(0,1)$ with $n=10$}
		\label{1-fig-N01C01-largest-order-statistics-distribution-functions-ggb}
	\end{subfigure}
	\caption{}
	\label{1-fig-distributionplot}
\end{figure}

We observe that, in the region $(-\infty,0)$, i.e. where $F_Y$ dominates $F_X$ (and hence $F_{Y_{n:n}}$ dominates $F_{X_{n:n}}$), the difference between $F_{X_{n:n}}$ and $F_{Y_{n:n}}$ is almost negligible, while in the region $(0,\infty)$, where the dominance is opposite, the difference is inflated. These observations motivate us to examine stochastic dominance that may exist between $X_{n:n}$ and $Y_{n:n}$ in some asymptotic sense. We find out that this is indeed the case, according to the asymptotic stochastic orders, formulated in this work.\vs

Throughout the paper, we stick to the following notations and conventions. For any function $f:\mathbb{R} \to \mathbb{R}$ and any set $A \subseteq \mathbb{R}$, we denote $f(A)=\{f(x) : x \in A\}$. For any $x \in \mathbb{R}$ and $S \subseteq \mathbb{R}$, we denote $\textnormal{dist}\left(x,S\right)=\inf_{y\,\in\,S} \left\vert x-y \right\vert$. For an arbitrary set $T \subseteq \mathbb{R}$ which is unbounded above, the limit of a real-valued function $t \mapsto \psi(t)$ defined on $T$, as $t \to \infty$, is considered in the usual sense, i.e. $\lim_{t \to \infty} \psi(t) \equiv \lim_{t \to \infty;\,t \,\in\, T} \psi(t)=a \in \mathbb{R} \cup \left\{-\infty,\infty\right\}$ if for every sequence $\left\{t_n \in T : n \in \mathbb{N}\right\}$ with $t_n \to \infty$, as $n \to \infty$, we have $\lim_{n \to \infty} \psi(t_n)=a$. For any real-valued function $g$, we define the left-continuous inverse by $g^{\leftarrow}(y)=\inf{\{x \in \mathbb{R} : g(x) \geq y\}}$ and the right-continuous inverse of $g$ by $g^{\rightarrow}(y)=\sup{\{x \in \mathbb{R} : g(x) \leq y\}}$. Evidently, $g^{\leftarrow}(y) \leq g^{\rightarrow}(y)$. We denote $g(x-)=\lim_{t \uparrow x} g(t)$ and $g(x+)=\lim_{t \downarrow x} g(t)$, when the limits exist.\vs

The rest of the paper is structured as follows. In \cref{1-section-ast} and \cref{1-section-asp}, respectively, we give the formal definitions of asymptotic usual stochastic order and asymptotic stochastic precedence order. Various properties and interrelation of these two orders are also discussed. In \cref{1-section-motivation-quantitative}, the need for quantitative approach in formulating asymptotic stochastic orders is illustrated and the basis of two quantitative asymptotic stochastic orders is developed, based on $L_1$ and $\mathcal{W}_2$ distances (defined in \cref{1-section-asymptotic-L-1} and \cref{1-section-asymptotic-W-2}, respectively) between cdfs. In \cref{1-section-asymptotic-L-1} and \cref{1-section-asymptotic-W-2}, respectively, we define and analyze the $L_1$-asymptotic usual stochastic order and the $\mathcal{W}_2$-asymptotic usual stochastic order. Interrelation between the two orders is also shown. In \cref{1-section-applications}, we have derived sufficient conditions for the asymptotic stochastic orders for mixtures of order statistics from samples with possibly non-integral size, as well as record values from two different homogeneous samples, as the sample size becomes large, to hold.\vs

\section{Asymptotic usual stochastic order}\label{1-section-ast}

Let $\left\{X_t : t \in T\right\}$ and $\left\{Y_t : t \in T\right\}$ be two stochastic processes with $T \subseteq \mathbb{R}$. Throughout the paper, we shall assume the index set $T$ to be unbounded above. Note that, when $T=\mathbb{N}$ (resp. $T=[0,\infty)$), the processes become discrete-time (resp. continuous-time) stochastic processes. It is a well-known fact that $X \leq_{\textnormal{st}} Y$ if and only if $F_X^{\leftarrow}(u) \leq F_Y^{\leftarrow}(u)$, for every $u \in (0,1)$. \citet{L_1955} notes that this characterization of usual stochastic order in terms of the quantile functions is more intuitive than \cref{1-definition-usual-stochastic-order}. Now, it may happen that there does not exist any $t \in T$ such that $X_t \leq_{\textnormal{st}} Y_t$ but, as $t \to \infty$, it increasingly gets closer to $X_t \leq_{\textnormal{st}} Y_t$, in the sense that $F_{X_t}^{\leftarrow}(u) \leq F_{Y_t}^{\leftarrow}(u)$ is satisfied for every $u \in (0,1)$, except for a set, whose Lebesgue measure shrinks to $0$, as $t \to \infty$. This leads to the following definition.

\begin{definition}\label{1-definition-asymptotic-usual-stochastic-order}
	Let $\left\{X_t : t \in T\right\}$ and $\left\{Y_t : t \in T\right\}$ be two stochastic processes, with respective classes of distribution functions $\left\{F_{X_t} : t \in T\right\}$ and $\left\{F_{Y_t} : t \in T\right\}$. We say that $X_t$ is smaller than $Y_t$ in asymptotic usual stochastic order, denoted by $X_t \leq_{\textnormal{ast}} Y_t$, as $t \to \infty$, if
	\begin{equation}\label{1-eq-definition-asymptotic-usual-stochastic-order}
	\lim_{t \to \infty} \mu\left(\left\{u \in (0,1) : F_{X_t}^{\leftarrow}(u) > F_{Y_t}^{\leftarrow}(u)\right\}\right)=0,
	\end{equation}
	where $\mu$ is the Lebesgue measure. Also, we say that $X_t$ is equal to $Y_t$ in asymptotic usual stochastic order and denote it by $X_t =_{\text{ast}} Y_t$, as $t \to \infty$, if both $X_t \leq_{\textnormal{ast}} Y_t$, as $t \to \infty$ and $Y_t \leq_{\textnormal{ast}} X_t$, as $t \to \infty$ hold true. \hfill $\blacksquare$
\end{definition}

The next lemma shows that the condition, stated in \eqref{1-eq-definition-asymptotic-usual-stochastic-order}, remains the same if any of the left-continuous inverses is changed to right-continuous inverse. The proof is straightforward and hence omitted.

\begin{lemma}\label{1-lemma-inverse-consistency}
	Let $F_X$ and $F_Y$ be two cdfs. Then the following quantities are equal.
	\begin{align*}
		&(i)\,\,\mu(\{u \in (0,1) : F_X^{\leftarrow}(u)>F_Y^{\leftarrow}(u)\}),\\
		&(ii)\,\,\mu(\{u \in (0,1) : F_X^{\leftarrow}(u)>F_Y^{\rightarrow}(u)\}),\\
		&(iii)\,\,\mu(\{u \in (0,1) : F_X^{\rightarrow}(u)>F_Y^{\rightarrow}(u)\}),\\
		&(iv)\,\,\mu(\{u \in (0,1) : F_X^{\rightarrow}(u)>F_Y^{\leftarrow}(u)\}). \tag*{$\blacksquare$}
	\end{align*}
\end{lemma}

If the respective distributions of $X_t$ and $Y_t$ are continuous, for every $t \in T$, then it is possible to characterize asymptotic usual stochastic order in terms of the probability measure assigned to the region of the support where $F_{X_t} \geq F_{Y_t}$, by both $P_{X_t}$ and $P_{Y_t}$, the probability measures induced by $F_{X_t}$ and $F_{Y_t}$, respectively. For that purpose, we need the following lemmas, the proofs of which are given in the appendix.

\begin{lemma}\label{1-lemma-quantile-distribution}
	Let $F_X$ and $F_Y$ be continuous cdfs and let $P_X$ be the probability measure induced by $F_X$. Then we have
	\[
	\mu(\{u \in (0,1) : F_X^{\rightarrow}(u) > F_Y^{\rightarrow}(u)\})=P_X\left(\{x \in \mathbb{R} : F_X(x) < F_Y(x)\}\right).
	\]
\end{lemma}

\begin{lemma}\label{1-lemma-no-inconsistency}
	Let $F_X$ and $F_Y$ be continuous cdfs with respective induced probability measures $P_X$ and $P_Y$. Then we have
	\[
	P_X\left(\{x \in \mathbb{R} : F_X(x) < F_Y(x)\}\right)=P_Y\left(\{x \in \mathbb{R} : F_X(x) < F_Y(x)\}\right).
	\]
\end{lemma}

\ni{\it Remark.} To see the necessity of the continuity assumption on $F_X$ and $F_Y$ for both \cref{1-lemma-quantile-distribution} and \cref{1-lemma-no-inconsistency} to hold, consider the cdfs $F_X$ and $F_Y$, defined as $F_X(x)=\frac{1}{2}1_{[-1,1)}(x)+1_{[1,\infty)}(x)$ and $F_Y(x)=\frac{1}{4}1_{[-2,0)}(x)+\frac{3}{4}1_{[0,2)}(x)+1_{[2,\infty)}(x)$, for every $x \in \mathbb{R}$, where $1_A(x)=1$, if $x \in A$ and $0$, otherwise, for every $A \subseteq \mathbb{R}$. \hfill $\blacksquare$\vs

The next theorem shows characterization of the asymptotic usual stochastic order in terms of probability measures of the region $\{x \in \mathbb{R} : F_{X_t}(x) < F_{Y_t}(x)\}$.

\begin{theorem}\label{1-theorem-ast-inverse}
	 Let $\left\{X_t : t \in T\right\}$ and $\left\{Y_t : t \in T\right\}$ be two stochastic processes, with respective classes of continuous cdfs $\left\{F_{X_t} : t \in T\right\}$ and $\left\{F_{Y_t} : t \in T\right\}$. Then the following statements are equivalent.
	\begin{enumerate}[label=\textnormal{(\roman*)}]
	\item $X_t \leq_{\textnormal{ast}} Y_t$, as $t \to \infty$.
	\item $\lim_{t \to \infty} P_{X_t}\left(\left\{x \in \mathbb{R} : F_{X_t}(x) < F_{Y_t}(x)\right\}\right)=0$.
	\item $\lim_{t \to \infty} P_{Y_t}\left(\left\{x \in \mathbb{R} : F_{X_t}(x) < F_{Y_t}(x)\right\}\right)=0$. \hfill $\blacksquare$
	\end{enumerate}
\end{theorem}

The proof follows from the lemmas \ref{1-lemma-inverse-consistency}, \ref{1-lemma-quantile-distribution} and \ref{1-lemma-no-inconsistency}. Next, we give a characterization of asymptotic equality in usual stochastic order in the next proposition. The proof trivially follows from \cref{1-definition-asymptotic-usual-stochastic-order} and is omitted.
\begin{proposition}\label{1-proposition-asymptotic-usual-stochastic-order-equality}
	$X_t =_{\textnormal{ast}} Y_t$, as $t \to \infty$ if and only if
	\[
	\lim_{t \to \infty} \mu(\{u \in (0,1) : F_{X_t}^{\leftarrow}(u)=F_{Y_t}^{\leftarrow}(u)\})=1. \tag*{$\blacksquare$}
	\]
\end{proposition}

The next result shows that the asymptotic usual stochastic order is a partial order.

\begin{theorem}\label{1-theorem-asymptotic-usual-stochastic-order-partial-order}
	Asymptotic usual stochastic order is a partial order.
\end{theorem}

\begin{proof}
	Asymptotic usual stochastic order is trivially reflexive and antisymmetric. Let $\left\{X_t : t \in T\right\}$, $\left\{Y_t : t \in T\right\}$ and $\left\{Z_t : t \in T\right\}$ be three stochastic processes. To prove transitivity, we have to show that $X_t \leq_{\textnormal{ast}} Y_t$, as $t \to \infty$ and $Y_t \leq_{\textnormal{ast}} Z_t$, as $t \to \infty$ imply that $X_t \leq_{\textnormal{ast}} Z_t$, as $t \to \infty$. Observe that, $F_{X_t}^{\leftarrow}(u) > F_{Z_t}^{\leftarrow}(u)$ implies that at least one of the conditions $F_{X_t}^{\leftarrow}(u) > F_{Y_t}^{\leftarrow}(u)$ and $F_{Y_t}^{\leftarrow}(u) > F_{Z_t}^{\leftarrow}(u)$ must hold, for otherwise we shall have $F_{X_t}^{\leftarrow}(u) \leq F_{Z_t}^{\leftarrow}(u)$, a contradiction. Let us denote $A_{0,t}=\{u \in (0,1) : F_{X_t}^{\leftarrow}(u) > F_{Y_t}^{\leftarrow}(u)\}$, $B_{0,t}=\{u \in (0,1) : F_{Y_t}^{\leftarrow}(u) > F_{Z_t}^{\leftarrow}(u)\}$ and $C_{0,t}=\{u \in (0,1) : F_{X_t}^{\leftarrow}(u) > F_{Z_t}^{\leftarrow}(u)\}$. Then $C_{0,t} \subseteq A_{0,t} \medcup B_{0,t}$. Thus, $\mu(C_{0,t}) \leq \mu(A_{0,t} \medcup B_{0,t}) \leq \mu(A_{0,t})+\mu(B_{0,t})$. Now, $\lim_{t \to \infty} \mu(A_{0,t})=0$ and $\lim_{t \to \infty} \mu(B_{0,t})=0$, due to the assumptions that $X_t \leq_{\textnormal{ast}} Y_t$, as $t \to \infty$ and $Y_t \leq_{\textnormal{ast}} Z_t$, as $t \to \infty$, respectively. Hence, $\lim_{t \to \infty} \mu(C_{0,t})=0$ and the proof follows.
\end{proof}

The next theorem shows that asymptotic usual stochastic order for stochastic processes is closed under nondecreasing, continuous transformations.

\begin{theorem}\label{1-theorem-ast-increasing-transformations}
	Let $\left\{X_t : t \in T\right\}$ and $\left\{Y_t : t \in T\right\}$ be two stochastic processes such that $X_t \leq_{\textnormal{ast}} Y_t$, as $t \to \infty$. Also, let $\psi:\mathbb{R} \to \mathbb{R}$ be a nondecreasing, continuous function. Then $\psi(X_t) \leq_{\textnormal{ast}} \psi(Y_t)$, as $t \to \infty$. Furthermore, if $\psi$ is strictly increasing, then $X_t \leq_{\textnormal{ast}} Y_t$, as $t \to \infty$ if and only if $\psi(X_t) \leq_{\textnormal{ast}} \psi(Y_t)$, as $t \to \infty$.
\end{theorem}

\begin{proof}
	Let us consider the situation where $\psi$ is nondecreasing. Then, for every $x \in \mathbb{R}$, we have $F_{\psi(X_t)}(x)=F_{X_t}(\psi^{\rightarrow}(x))$ and $F_{\psi(Y_t)}(x)=F_{Y_t}(\psi^{\rightarrow}(x))$. Thus, for every $u \in (0,1)$, we have $F_{\psi(X_t)}^{\rightarrow}(u)=\psi(F_{X_t}^{\rightarrow}(u))$ and $F_{\psi(Y_t)}^{\rightarrow}(u)=\psi(F_{Y_t}^{\rightarrow}(u))$. Since $\psi$ is nondecreasing, we have
	\begin{align}\label{1-eq-theorem-ast-increasing-transformations}
	&\phantom{\,\,\,\,\,\,\,\,}\mu\left(\left\{u \in \left(0,1\right) : F_{\psi(X_t)}^{\rightarrow}(u)>F_{\psi(Y_t)}^{\rightarrow}(u)\right\}\right)\nonumber\\
	&=\mu\left(\left\{u \in \left(0,1\right) : \psi\left(F_{X_t}^{\rightarrow}(u)\right)>\psi\left(F_{Y_t}^{\rightarrow}(u)\right)\right\}\right)\nonumber\\
	&=1-\mu\left(\left\{u \in \left(0,1\right) : \psi\left(F_{X_t}^{\rightarrow}(u)\right) \leq \psi\left(F_{Y_t}^{\rightarrow}(u)\right)\right\}\right)\nonumber\\
	&\leq 1-\mu\left(\left\{u \in \left(0,1\right) : F_{X_t}^{\rightarrow}(u) \leq F_{Y_t}^{\rightarrow}(u)\right\}\right)\\
	&=\mu\left(\left\{u \in \left(0,1\right) : F_{X_t}^{\rightarrow}(u)>F_{Y_t}^{\rightarrow}(u)\right\}\right).\nonumber
	\end{align}
	By taking limit as $t \to \infty$ and applying \cref{1-lemma-inverse-consistency}, it follows from \cref{1-definition-asymptotic-usual-stochastic-order} that $\psi(X_t) \leq_{\textnormal{ast}} \psi(Y_t)$, as $t \to \infty$. If $\psi$ is strictly increasing, then \eqref{1-eq-theorem-ast-increasing-transformations} becomes an equality, which implies that $X_t \leq_{\textnormal{ast}} Y_t$, as $t \to \infty$ if and only if $\psi(X_t) \leq_{\textnormal{ast}} \psi(Y_t)$, as $t \to \infty$.
\end{proof}

The next lemma shows that asymptotic usual stochastic order is reversed by negation of stochastic processes. The proof follows from \cref{1-lemma-inverse-consistency} and the observations that $F_{-X}^{\leftarrow}(u)=-F_X^{\rightarrow}(1-u)$ and $F_{-X}^{\rightarrow}(u)=-F_X^{\leftarrow}(1-u)$, for every $u \in (0,1)$, for any random variable $X$.

\begin{lemma}\label{1-lemma-ast-negation-2}
	Let $\left\{X_t : t \in T\right\}$ and $\left\{Y_t : t \in T\right\}$ be two stochastic processes such that $X_t \leq_{\textnormal{ast}} Y_t$, as $t \to \infty$. Then $-Y_t \leq_{\textnormal{ast}} -X_t$, as $t \to \infty$. \hfill $\blacksquare$
\end{lemma}

Observing that if $\psi$ is nonincreasing, then $-\psi$ is nondecreasing and using \cref{1-lemma-ast-negation-2} we have the following corollary of \cref{1-theorem-ast-increasing-transformations}.

\begin{corollary}\label{1-corollary-ast-decreasing-transformations}
	Let $\left\{X_t : t \in T\right\}$ and $\left\{Y_t : t \in T\right\}$ be two stochastic processes such that $X_t \leq_{\textnormal{ast}} Y_t$, as $t \to \infty$. Also, let $\psi:\mathbb{R} \to \mathbb{R}$ be a nonincreasing,
	continuous function. Then $\psi(Y_t) \leq_{\textnormal{ast}} \psi(X_t)$,
	as $t \to \infty$. Furthermore, if $\psi$ is strictly decreasing, then $X_t \leq_{\textnormal{ast}} Y_t$, as $t \to \infty$ if and only if $\psi(Y_t) \leq_{\textnormal{ast}} \psi(X_t)$, as $t \to \infty$. \hfill $\blacksquare$
\end{corollary}

In particular, for linear transformations of stochastic processes, we have the following result.

\begin{corollary}\label{1-corollary-ast-linear-transformations}
	Let $\left\{X_t : t \in T\right\}$ and $\left\{Y_t : t \in T\right\}$ be two stochastic processes such that $X_t \leq_{\textnormal{ast}} Y_t$, as $t \to \infty$. Then $a+bX_t \leq_{\textnormal{ast}} a+bY_t$ \textnormal{(}resp. $a+bY_t \leq_{\textnormal{ast}} a+bX_t$\textnormal{)}, as $t \to \infty$, for every $a \in \mathbb{R}$ and $b \geq 0$ \textnormal{(}resp. $b \leq 0$\textnormal{)}. \hfill $\blacksquare$
\end{corollary}

The next theorem shows an intuitive connection between asymptotic usual stochastic order and convergence in probability of stochastic processes to certain subsets of $\mathbb{R}$.

\begin{theorem}\label{1-theorem-ast-convergence-in probability}
	Let $\left\{X_t : t \in T\right\}$ and $\left\{Y_t : t \in T\right\}$ be two stochastic processes such that $P_{X_t}(A) \to 1$, as $t \to \infty$ and $P_{Y_t}(B) \to 1$, as $t \to \infty$, where $A$ and $B$ are two subsets of $\mathbb{R}$ with $\sup{A}<\inf{B}$. Then $X_t \leq_{\textnormal{ast}} Y_t$, as $t \to \infty$.
\end{theorem}

\begin{proof}
	Let $0<\epsilon<1$ and $0<\delta<(\inf{B}-\sup{A})/2$. By the hypothesis, there exists $T_1, T_2 \in T$ such that $P_{X_t}(A)>1-\epsilon/2$, whenever $t \geq T_1$ and $P_{Y_t}(B)>1-\epsilon/2$, whenever $t \geq T_2$. Let us choose $t \geq \max{\{T_1,T_2\}}$ and $u \in [\epsilon/2,1-\epsilon/2]$. Then $\sup{A}+\delta \geq F_{X_t}^{\rightarrow}(u)$ and $\inf{B}-\delta<F_{Y_t}^{\leftarrow}(u)$. By choice of $\delta$, we have $\sup{A}+\delta<\inf{B}-\delta$. Hence, $F_{X_t}^{\rightarrow}(u)<F_{Y_t}^{\leftarrow}(u)$. Since this holds for every $u \in [\epsilon/2,1-\epsilon/2]$, we have $\mu(\{u \in (0,1) : F_{X_t}^{\rightarrow}(u)<F_{Y_t}^{\leftarrow}(u)\}) \geq 1-\epsilon$. Thus,
	\begin{align*}
	\mu(\{u \in (0,1) : F_{X_t}^{\rightarrow}(u)>F_{Y_t}^{\leftarrow}(u)\}) \leq 1-\mu(\{u \in (0,1) : F_{X_t}^{\rightarrow}(u)<F_{Y_t}^{\leftarrow}(u)\}) \leq \epsilon.
	\end{align*}
	Since $\epsilon \in (0,1)$ can be arbitrarily small, the proof follows from \cref{1-lemma-inverse-consistency}.
\end{proof}

\section{Asymptotic stochastic precedence order}\label{1-section-asp}

In the previous sections, we have explored asymptotic stochastic orders based on stochastic orders that do not consider any dependence structure between $X_t$ and $Y_t$ into account. In this section, we discuss an alternative approach based on stochastic precedence order, that considers any possible dependence structure between the two stochastic processes into account.

\begin{definition}\label{1-definition-asymptotic-stochastic-precedence-order}
	For two stochastic processes $\left\{X_t : t \in T\right\}$ and $\left\{Y_t : t \in T\right\}$, we say that $X_t$ is smaller than $Y_t$ in asymptotic stochastic precedence order, denoted by $X_t \leq_{\textnormal{asp}} Y_t$, as $t \to \infty$ if
	\begin{equation}\label{1-eq-definition-asymptotic-stochastic-precedence-order}
	\lim_{t \to \infty} P\{X_t \leq Y_t\} \geq 1/2.
	\end{equation}
	We say that $X_t$ is equal to $Y_t$ in asymptotic stochastic precedence order and denote it by $X_t =_{\textnormal{asp}} Y_t$, as $t \to \infty$ if $X_t \leq_{\textnormal{asp}} Y_t$, as $t \to \infty$ and $Y_t \leq_{\textnormal{asp}} X_t$, as $t \to \infty$.
\end{definition}

Note that, $X_t =_{\textnormal{asp}} Y_t$, as $t \to \infty$ if and only if $\lim_{t \to \infty} P\{X_t \leq Y_t\} = 1/2$. The next theorem shows that, for any two independent stochastic processes, asymptotic usual stochastic order is stronger than asymptotic stochastic precedence order. We need the next lemma to prove it.

\begin{lemma}\label{1-lemma-ast-asp}
	Let $F$ and $G$ be two cdfs. Then
	\begin{equation}\label{1-eq-lemma-ast-asp}
		\left\{G(x) : F(x)<G(x)\right\} \subseteq \left\{u \in (0,1) : F^{\leftarrow}(u)>G^{\leftarrow}(u)\right\}.
	\end{equation}
	Furthermore, if $G$ is continuous, then the two sets in \eqref{1-eq-lemma-ast-asp} become equal.
\end{lemma}

\ni{\it Remark.} To demonstrate that the reverse direction of \eqref{1-eq-lemma-ast-asp} does not hold in general, we consider $F$ to be the cdf of standard normal distribution and $G$ to be the cdf of the degenerate distribution on $0$. Then $\{G(x) : F(x)<G(x)\}=G([0,\infty))=\{1\}$ and $\{u \in (0,1) : F^{\leftarrow}(u)>G^{\leftarrow}(u)\}=(1/2,1]$.

\begin{theorem}\label{1-theorem-ast-asp}
	Let $\left\{X_t : t \in T\right\}$ and $\left\{Y_t : t \in T\right\}$ be two independent stochastic processes. If $X_t \leq_{\textnormal{ast}} Y_t$, as $t \to \infty$, then $X_t \leq_{\textnormal{asp}} Y_t$, as $t \to \infty$.
\end{theorem}

\begin{proof}
	Let $t \in T$ and define $P_{0,t}=\left\{x \in \mathbb{R} : F_{X_t}(x)<F_{Y_t}(x)\right\}$. Then, by independence of $X_t$ and $Y_t$, we have
	\begin{align}\label{1-proof-theorem-ast-asp}
	P\left\{X_t \leq Y_t\right\}&=\int_{-\infty}^{\infty} F_{X_t}(x)\,dF_{Y_t}(x)\nonumber\\
	&=\frac{1}{2}+\int_{P_{0,t}^c} \left(F_{X_t}(x)-F_{Y_t}(x)\right)\,dF_{Y_t}(x)-\int_{P_{0,t}} \left(F_{Y_t}(x)-F_{X_t}(x)\right)\,dF_{Y_t}(x)\nonumber\\
	&\geq \frac{1}{2}-\mu\left(\left\{F_{Y_t}(x) : F_{X_t}(x)<F_{Y_t}(x)\right\}\right)\nonumber\\
	&\geq \frac{1}{2}-\mu\left(\left\{u \in (0,1) : F_{X_t}^{\leftarrow}(u)>F_{Y_t}^{\leftarrow}(u)\right\}\right),
	\end{align}
	where the last inequality is due to \cref{1-lemma-ast-asp}. The proof then follows from the hypothesis.
\end{proof}

The next theorem shows that asymptotic stochastic precedence order of the stochastic processes is closed under nondecreasing transformations.

\begin{theorem}\label{1-theorem-asp-increasing-transformations}
	Let $\left\{X_t : t \in T\right\}$ and $\left\{Y_t : t \in T\right\}$ be two stochastic processes such that $X_t \leq_{\textnormal{asp}} Y_t$, as $t \to \infty$ and let $\psi:\mathbb{R} \to \mathbb{R}$ be a nondecreasing function. Then $\psi(X_t) \leq_{\textnormal{asp}} \psi(Y_t)$, as $t \to \infty$. Furthermore, if $\psi$ is strictly increasing, then $X_t \leq_{\textnormal{asp}} Y_t$, as $t \to \infty$ if and only if $\psi(X_t) \leq_{\textnormal{asp}} \psi(Y_t)$, as $t \to \infty$.
\end{theorem}

\begin{proof}
	By nondecreasingness of $\psi$, we have, for every $t \in T$, $X_t \leq Y_t$ implies $\psi(X_t) \leq \psi(Y_t)$ and hence $P\left\{\psi(X_t) \leq \psi(Y_t)\right\} \geq P\left\{X_t \leq Y_t\right\}$. Taking limit as $t \to \infty$ and using the hypothesis that $X_t \leq_{\textnormal{asp}} Y_t$, as $t \to \infty$, we have $\psi(X_t) \leq_{\textnormal{asp}} \psi(Y_t)$, as $t \to \infty$. For the second part, observe that if $\psi$ is strictly increasing, then for every $t \in T$, we have $P\left\{\psi(X_t) \leq \psi(Y_t)\right\}=P\left\{X_t \leq Y_t\right\}$. The proof follows by taking limit as $t \to \infty$.
\end{proof}

Observing that $-\psi$ is nondecreasing (resp. strictly increasing) if and only if $\psi$ is nonincreasing (resp. strictly decreasing), we immediately have the following corollary.

\begin{corollary}\label{1-corollary-asp-decreasing-transformations}
	Let $\left\{X_t : t \in T\right\}$ and $\left\{Y_t : t \in T\right\}$ be two stochastic processes such that $X_t \leq_{\textnormal{asp}} Y_t$, as $t \to \infty$ and let $\psi:\mathbb{R} \to \mathbb{R}$ be a nonincreasing function. Then $\psi(Y_t) \leq_{\textnormal{asp}} \psi(X_t)$, as $t \to \infty$. Furthermore, if $\psi$ is strictly decreasing, then $X_t \leq_{\textnormal{asp}} Y_t$, as $t \to \infty$ if and only if $\psi(Y_t) \leq_{\textnormal{asp}} \psi(X_t)$, as $t \to \infty$. \hfill $\blacksquare$
\end{corollary}

It follows from \cref{1-theorem-ast-asp} that under the setup of \cref{1-theorem-ast-convergence-in probability}, we have $X_t \leq_{\textnormal{asp}} Y_t$, as $t \to \infty$. The next theorem gives a stronger conclusion.

\begin{theorem}\label{1-theorem-asp-convergence-in probability}
	Let $\left\{X_t : t \in T\right\}$ and $\left\{Y_t : t \in T\right\}$ be two independent stochastic processes such that $P_{X_t}(A) \to 1$, as $t \to \infty$ and $P_{Y_t}(B) \to 1$, as $t \to \infty$, where $A$ and $B$ are two subsets of $\mathbb{R}$ with $\sup{A}<\inf{B}$. Then $\lim_{t \to \infty} P\{X_t \leq Y_t\}=1$.
\end{theorem}

\begin{proof}
	Let $\epsilon \in (0,1)$ and $\delta \in (0,(\inf{B}-\sup{A})/2)$. By assumption, there exists $T_1 \in T$ such that for every $t \geq T_1$, $F_{X_t}(x)<\epsilon$ if $x<\inf{A}-\delta$ and $F_{X_t}(x)>1-\epsilon$ if $x>\sup{A}+\delta$. Again, there exists $T_2 \in T$ such that for every $t \geq T_2$, $F_{Y_t}(x)<\epsilon$ if $x<\inf{B}-\delta$ and $F_{Y_t}(x)>1-\epsilon$ if $x>\sup{B}+\delta$. Let us choose $t \geq \max{\{T_1,T_2\}}$. Then
	\[
	P\{X_t \leq Y_t\}=\int_{-\infty}^{\infty} F_{X_t}(x)\,dF_{Y_t}(x) \geq \int_{\inf{B}-\delta}^{\sup{B}+\delta} F_{X_t}(x)\,dF_{Y_t}(x) \geq (1-\epsilon) \int_{\inf{B}-\delta}^{\sup{B}+\delta} dF_{Y_t}(x).
	\]
	Since, $F_{Y_t}(\inf{B}-\delta) \leq \epsilon$ and $F_{Y_t}(\sup{B}+\delta) \geq 1-\epsilon$, we have $P\{X_t \leq Y_t\} \geq (1-\epsilon)(1-2\epsilon) > 1-3\epsilon$. Since $\epsilon$ can be arbitrarily small, it follows that $\lim_{t \to \infty} P\{X_t \leq Y_t\}=1$.
\end{proof}

\section{Motivation for quantitative approaches}\label{1-section-motivation-quantitative}

As defined in \cref{1-section-ast}, $X_t \leq_{\textnormal{ast}} Y_t$, as $t \to \infty$ if $\lim_{t \to \infty} \mu(\{u \in (0,1) : F_{X_t}^{\leftarrow}(u) > F_{Y_t}^{\leftarrow}(u)\})=0$. A limitation of this condition is that, the measure $\mu(\{u \in (0,1) : F_{X_t}^{\leftarrow}(u) > F_{Y_t}^{\leftarrow}(u)\})$ considers only the sign of $F_{X_t}^{\leftarrow}(u)-F_{Y_t}^{\leftarrow}(u)$ and does not take into account the magnitude of the same. For that reason, we say that the definition is {\it semi-quantitative} in nature. As a result, $X_t \leq_{\textnormal{ast}} Y_t$, as $t \to \infty$ does not hold in a lot of situations where asymptotic stochastic order should hold from an intuitive point of view, as demonstrated in the following example.

\begin{example}\label{1-example-ast-L-1}
	Let us consider the stochastic processes $\{X_t : t \geq 0\}$ and $\{Y_t : t \geq 0\}$ with corresponding cdfs (of $X_t$ and $Y_t$)
	\begin{flalign*}
		&\hspace*{3.5cm} F_{X_t}(x)=
		\begin{cases*}
			\Phi\left(x\right), & if  $x<0$,\\
			\frac{1}{2}+\left(\Phi\left(1\right)-\frac{1}{2}-\frac{1}{t+4}\right)x, & if $0 \leq x<1$,\\
			1-e^{1-x}\left(1-\Phi\left(x\right)\right), & if  $x \geq 1$,
		\end{cases*}\\
		&\noindent \textnormal{and}&\\
		&\hspace*{3.5cm} F_{Y_t}(x)=
		\begin{cases*}
			e^x\,\Phi\left(x\right), & if  $x<0$,\\
			\frac{1}{2}+\left(\Phi\left(1\right)-\frac{1}{2}\right)x, & if $0 \leq x<1$,\\
			\Phi\left(x\right), & if  $x \geq 1$,
		\end{cases*}
	\end{flalign*}
	for every $t \geq 0$, where $\Phi$ denotes the cdf of the standard normal distribution. Let $A_{0,t}=\{u \in (0,1) : F_{X_t}^{\leftarrow}(u) > F_{Y_t}^{\leftarrow}(u)\}$. Then it is easy to see that $\lim_{t \to \infty} \mu(A_{0,t})=\Phi(1)-0.5>0$. Thus, asymptotic usual stochastic order does not hold. However, for every $t \geq 0$, we have $F_{X_t}^{\leftarrow}(u) \leq F_{Y_t}^{\leftarrow}(u)$, whenever $u \notin A_{0,t}$ and $\lim_{t \to \infty} F_{X_t}^{\leftarrow}(u)=\lim_{t \to \infty} F_{Y_t}^{\leftarrow}(u)$, whenever $u \in A_{0,t}$, i.e. $F_{X_t}^{\leftarrow}(u)$ and $F_{Y_t}^{\leftarrow}(u)$ are arbitrarily close, if $u \in A_{0,t}$, when $t$ is large enough. This suggests that the usual stochastic order should hold in some asymptotic sense, but \cref{1-definition-asymptotic-usual-stochastic-order} does not reflect that intuition. \hfill $\blacksquare$
\end{example}

Also, $X_t \leq_{\textnormal{ast}} Y_t$, as $t \to \infty$ does not imply $\lim_{t \to \infty} E\left(X_t\right) \leq \lim_{t \to \infty} E\left(Y_t\right)$, when the limits exist. See the remark following \cref{1-theorem-ast-expectation-inequality-L-1}, on construction of examples where $X_t \leq_{\textnormal{ast}} Y_t$, as $t \to \infty$ and $\lim_{t \to \infty} E\left(X_t\right)>\lim_{t \to \infty} E\left(Y_t\right)$.\vs

To overcome these issues, it is clear that one has to incorporate not only the region where $X \leq_{\textnormal{st}} Y$ is violated, but also the extent of that violation, in the definition of asymptotic usual stochastic order. Let $\Vert \cdot \Vert_{L_1}$ be the $L_1$ norm with respect to the Lebesgue measure on $\mathbb{R}$. Now, we can partition the $L_1$ distance between two cdfs $F_X$ and $F_Y$ in the following way.
\begin{equation}\label{1-eq-decomposition-L-1}
	\Vert F_Y-F_X \Vert_{L_1}=\int_{P_0} \left\vert F_X(x)-F_Y(x) \right\vert dx+\int_{P_1} \left\vert F_X(x)-F_Y(x) \right\vert dx,
\end{equation}
where $P_0=\{x \in \mathbb{R} : F_X(x)<F_Y(x)\}$ and $P_1=\{x \in \mathbb{R} : F_X(x)>F_Y(x)\}$. Note that the first term in the right hand side vanishes if $X \leq_{\textnormal{st}} Y$ and reaches the maximum $\Vert F_Y-F_X \Vert_{L_1}$ if $Y \geq_{\textnormal{st}} X$. Similarly, the second term in the right hand side matches $\Vert F_Y-F_X \Vert_{L_1}$ if $X \leq_{\textnormal{st}} Y$ and vanishes if $Y \geq_{\textnormal{st}} X$. Hence, roughly speaking, the first term contributes against $X \leq_{\textnormal{st}} Y$ and the second term contributes towards $X \leq_{\textnormal{st}} Y$.
Let $\Vert F_Y-F_X \Vert_{\mathcal{W}_p}$ be the $L_p$-Wasserstein distance (also known as $\mathcal{W}_p$ distance) between $F_X$ and $F_Y$. This distance has the following representation in terms of the corresponding quantile functions $F_X^{\leftarrow}$ and $F_Y^{\leftarrow}$.
\begin{equation*}
	\Vert F_Y-F_X \Vert_{\mathcal{W}_p}=\left(\int_0^1 (F_X^{\leftarrow}(u)-F_Y^{\leftarrow}(u))^p du\right)^{1/p}.
\end{equation*}
See \citet{PZ_2019} for a detailed review on the Wasserstein distances and their statistical aspects. It is well-known that if $F_X$ and $F_Y$ both have finite $p$th order moment, then $\Vert F_Y-F_X \Vert_{\mathcal{W}_p}$ is finite. Now, consider the following partition of $\Vert F_Y-F_X \Vert_{\mathcal{W}_2}^2$.
\begin{equation}\label{1-eq-decomposition-W-2}
	\Vert F_Y-F_X \Vert_{\mathcal{W}_2}^2=\int_{A_0} (F_X^{\leftarrow}(u)-F_Y^{\leftarrow}(u))^2\,du+\int_{A_1} (F_X^{\leftarrow}(u)-F_Y^{\leftarrow}(u))^2\,du,
\end{equation}
where $A_0=\left\{u \in (0,1) : F_X^{\leftarrow}(u)>F_Y^{\leftarrow}(u)\right\}$ and $A_1=\left\{u \in (0,1) : F_X^{\leftarrow}(u)<F_Y^{\leftarrow}(u)\right\}$. In a similar line of argument as in the situation involving $L_1$ distance, one can argue that the first term contributes against $X \leq_{\textnormal{st}} Y$, while the second term contributes towards $X \leq_{\textnormal{st}} Y$. In the next two sections, we define and analyze two asymptotic usual stochastic orders, based on the decompositions given in \eqref{1-eq-decomposition-L-1} and \eqref{1-eq-decomposition-W-2}, respectively.\vs

\section{$L_1$-asymptotic usual stochastic order}\label{1-section-asymptotic-L-1}

In this section, we consider an asymptotic stochastic order, based on the $L_1$ distance between cdfs (also called Kantorovich distance). We define it as follows.

\begin{definition}\label{1-definition-asymptotic-usual-stochastic-order-L-1}
	Let $\left\{X_t : t \in T\right\}$ and $\left\{Y_t : t \in T\right\}$ be two stochastic processes, with respective classes of cdfs $\left\{F_{X_t} : t \in T\right\}$ and $\left\{F_{Y_t} : t \in T\right\}$. We say that $X_t$ is smaller than $Y_t$ in $L_1$-asymptotic usual stochastic order, denoted by $X_t \leq_{L_1\textnormal{-ast}} Y_t$, as $t \to \infty$, if
	\begin{equation}\label{1-eq-definition-asymptotic-usual-stochastic-order-L-1}
		\lim_{t \to \infty} \int_{P_{0,t}} \left\vert F_{X_t}(x)-F_{Y_t}(x) \right\vert dx=0,
	\end{equation}
	where $P_{0,t}=\{x \in \mathbb{R} : F_{X_t}(x)<F_{Y_t}(x)\}$. We say that $X_t$ is equal to $Y_t$ in $L_1$-asymptotic usual stochastic order and denote it by $X_t =_{L_1\text{-ast}} Y_t$, as $t \to \infty$, if both $X_t \leq_{L_1\textnormal{-ast}} Y_t$ and $Y_t \leq_{L_1\textnormal{-ast}} X_t$ hold true, as $t \to \infty$. \hfill $\blacksquare$
\end{definition}

The $L_1$-asymptotic usual stochastic order neither implies, nor is implied by the asymptotic usual stochastic order, discussed in \cref{1-section-ast}. The next two examples demonstrate this point.

\begin{example}\label{1-example-ast-L-1-1}
	Let us consider the stochastic processes $\{X_t : t \geq 0\}$ and $\{Y_t : t \geq 0\}$ with corresponding cdfs (of $X_t$ and $Y_t$)
	\begin{flalign*}
		&\hspace*{2cm} F_{X_t}(x)=
		\begin{cases*}
			\Phi\left(x\right), & if  $x<0$,\\
			\frac{1}{2}, & if $0 \leq x<1$,\\
			1-e^{1-x}\left(1-\Phi\left(x\right)\right), & if  $x \geq 1$,
		\end{cases*}\\
		&\textnormal{and} &\\
		&\hspace*{2cm} F_{Y_t}(x)=
		\begin{cases*}
			e^x\,\Phi\left(x\right), & if  $x<0$,\\
			\frac{1}{2}, & if $0 \leq x<1-\frac{1}{t+1}$,\\
			\frac{1}{2}+\left(\Phi\left(1\right)-\frac{1}{2}\right)\{1+\left(t+1\right)\left(x-1\right)\}, & if $1-\frac{1}{t+1} \leq x<1$,\\
			\Phi\left(x\right), & if  $x \geq 1$,
		\end{cases*}
	\end{flalign*}
	for every $t \geq 0$. It is easy to verify that $\mu\left(\left\{u \in (0,1) : F_{X_t}^{\leftarrow}(u) > F_{Y_t}^{\leftarrow}(u)\right\}\right)=\Phi(1)-0.5$ and $\int_{P_{0,t}} \left\vert F_{X_t}(x)-F_{Y_t}(x) \right\vert dx=(\Phi(1)-0.5)/(2t+2)$, for every $t \geq 0$. It follows that the asymptotic usual stochastic order does not hold, but the $L_1$-asymptotic usual stochastic order holds.
\end{example}

\begin{example}\label{1-example-ast-L-1-2}
	Let us consider the stochastic processes $\{X_t : t \geq 0\}$ and $\{Y_t : t \geq 0\}$ with corresponding cdfs (of $X_t$ and $Y_t$)
	\begin{flalign*}
		&\hspace*{3.7cm} F_{X_t}(x)=
		\begin{cases*}
			\Phi\left(x+t\right), & if  $x<-t$,\\
			\frac{1}{2}, & if $-t \leq x<1$,\\
			1-e^{1-x}\left(1-\Phi\left(x\right)\right), & if  $x \geq 1$,
		\end{cases*}\\
		&\textnormal{and} &\\
		&\hspace*{3.7cm} F_{Y_t}(x)=
		\begin{cases*}
			e^{x+t}\,\Phi\left(x+t\right), & if  $x<-t$,\\
			\frac{1}{2}+\left\{\frac{\Phi(1)-\frac{1}{2}}{(t+1)^2}\right\}(x+t), & if $-t \leq x<1$,\\
			\Phi\left(x\right), & if  $x \geq 1$,
		\end{cases*}
	\end{flalign*}
	for every $t \geq 0$. It is easy to verify that asymptotic usual stochastic order holds, but $L_1$-asymptotic usual stochastic order does not. \hfill $\blacksquare$
\end{example}

Now, we shall discuss some properties of the $L_1$-asymptotic usual stochastic order. The following proposition shows that equality of $X_t$ and $Y_t$ in $L_1$-asymptotic usual stochastic order is characterized by convergence of the $L_1$ distance between $X_t$ and $Y_t$ to $0$, as $t \to \infty$. The proof is trivial and hence omitted.

\begin{proposition}\label{1-proposition-equality-asymptotic-usual-stochastic-order-L-1}
	$X_t =_{L_1\textnormal{-ast}} Y_t$, as $t \to \infty$ if and only if $\left\Vert F_{X_t}-F_{Y_t} \right\Vert_{L_1} \to 0$, as $t \to \infty$. \hfill $\blacksquare$
\end{proposition}

The next result shows that $L_1$-asymptotic stochastic order is a partial order.

\begin{theorem}\label{1-theorem-transitivity-L-1-asymptotic-usual-stochastic-order}
	$L_1$-asymptotic usual stochastic order is a partial order.
\end{theorem}

\begin{proof}
	$L_1$-asymptotic usual stochastic order is trivially {\it reflexive} and {\it antisymmetric}. Let $\left\{X_t : t \in T\right\}$, $\left\{Y_t : t \in T\right\}$ and $\left\{Z_t : t \in T\right\}$ be three stochastic processes such that $X_t \leq_{L_1\textnormal{-ast}} Y_t$, as $t \to \infty$ and $Y_t \leq_{L_1\textnormal{-ast}} Z_t$, as $t \to \infty$. To prove {\it transitivity}, we have to show that $X_t \leq_{L_1\textnormal{-ast}} Z_t$, as $t \to \infty$. Let $P_{0,t}=\left\{x \in \mathbb{R} : F_{X_t}(x)<F_{Y_t}(x)\right\}$, $Q_{0,t}=\left\{x \in \mathbb{R} : F_{Y_t}(x)<F_{Z_t}(x)\right\}$ and $R_{0,t}=\left\{x \in \mathbb{R} : F_{X_t}(x)<F_{Z_t}(x)\right\}$. Then $R_{0,t} \subseteq P_{0,t} \medcup Q_{0,t}$. Now, we consider the partition $R_{0,t}=S_1 \medcup S_2 \medcup S_3$, where $S_1=P_{0,t} \medcap Q_{0,t}^c \medcap R_{0,t}$, $S_2=P_{0,t}^c \medcap Q_{0,t} \medcap R_{0,t}$ and $S_3=P_{0,t} \medcap Q_{0,t} \medcap R_{0,t}$. Then
	\begin{equation}\label{1-eq-theorem-transitivity-asymptotic-usual-stochastic-order-L-1}
		\int_{R_{0,t}} \left\vert F_{X_t}(x)-F_{Z_t}(x) \right\vert dx = \sum_{i=1}^3 \int_{S_i} \left\vert F_{X_t}(x)-F_{Z_t}(x) \right\vert dx.
	\end{equation}
	Note that, for $x \in S_1$, we have
	\begin{equation*}
		\int_{S_1} \left\vert F_{X_t}(x)-F_{Z_t}(x) \right\vert dx \leq \int_{P_{0,t}} \left\vert F_{X_t}(x)-F_{Y_t}(x) \right\vert dx.
	\end{equation*}
	Similarly, we can see that
	\begin{equation*}
		\int_{S_2} \left\vert F_{X_t}(x)-F_{Z_t}(x) \right\vert dx \leq \int_{Q_{0,t}} \left\vert F_{Y_t}(x)-F_{Z_t}(x) \right\vert dx
	\end{equation*}
	and
	\[
	\int_{S_3} \left\vert F_{X_t}(x)-F_{Z_t}(x) \right\vert dx \leq \int_{P_{0,t}} \left\vert F_{X_t}(x)-F_{Y_t}(x) \right\vert dx+\int_{Q_{0,t}} \left\vert F_{Y_t}(x)-F_{Z_t}(x) \right\vert dx.
	\]
	Applying these upper bounds in \eqref{1-eq-theorem-transitivity-asymptotic-usual-stochastic-order-L-1}, we have
	\[
	\int_{R_{0,t}} \left\vert F_{X_t}(x)-F_{Z_t}(x) \right\vert dx \leq 2\left\{\int_{P_{0,t}} \left\vert F_{X_t}(x)-F_{Y_t}(x) \right\vert dx+\int_{Q_{0,t}} \left\vert F_{Y_t}(x)-F_{Z_t}(x) \right\vert dx\right\}.
	\]
	The proof then follows from the hypotheses.
\end{proof}

Let us define $T_g=\{y \in \mathbb{R} : g^{\leftarrow}(y)<g^{\rightarrow}(y)\}$ for any real-valued function $g$. Observe that $T_g$ is at most countable if $g$ is nondecreasing. The following result shows that $L_1$-asymptotic usual stochastic order is closed under nondecreasing transformations under certain conditions.

\begin{theorem}\label{1-theorem-increasing-function-asymptotic-usual-stochastic-order-L-1}
	Let $\left\{X_t : t \in T\right\}$ and $\left\{Y_t : t \in T\right\}$ be two stochastic processes such that $X_t \leq_{L_1\textnormal{-ast}} Y_t$, as $t \to \infty$. Also, let $\psi:\mathbb{R} \to \mathbb{R}$ be a nondecreasing function, which is differentiable almost everywhere with bounded derivative. Then $\psi(X_t) \leq_{L_1\textnormal{-ast}} \psi(Y_t)$, as $t \to \infty$.
\end{theorem}

\begin{proof}
	First, we consider the situation where $\psi$ is nondecreasing. Let $C_{\psi}$ be the set of points where $\psi$ is continuous and let $D_{\psi}$ be the set of points where $\psi$ is differentiable. Clearly, $D_{\psi} \subseteq C_{\psi}$ and by the hypothesis, $\mu(D_{\psi}^c)=0$. We have $F_{\psi(X_t)}(x)=F_{X_t}(\psi^{\rightarrow}(x))$ and $F_{\psi(Y_t)}(x)=F_{Y_t}(\psi^{\rightarrow}(x))$, for every $x \in C_{\psi}$. Again, by the hypothesis, $\psi$ has bounded derivative on $D_{\psi}$. Below we take $K$ as an upper bound of $\psi'(x)$ over $D_{\psi}$. Let $P_{0,t}=\{x \in \mathbb{R} : F_{X_t}(x)<F_{Y_t}(x)\}$ and $P_{0,t,\psi}=\{x \in \mathbb{R} : F_{\psi(X_t)}(x)<F_{\psi(Y_t)}(x)\}$. Since, $P_{0,t,\psi} \setminus C_{\psi} \subseteq P_{0,t,\psi} \setminus D_{\psi} \subseteq D_{\psi}^c$, we have
	\begin{align}\label{1-eq-theorem-increasing-function-asymptotic-usual-stochastic-order-L-1-1-C-psi}
		\int_{P_{0,t,\psi}} \left\vert F_{\psi(X_t)}\left(x\right)-F_{\psi(Y_t)}\left(x\right) \right\vert dx &=\int_{P_{0,t,\psi} \medcap C_{\psi}} \left\vert F_{X_t}\left(\psi^{\rightarrow}\left(x\right)\right)-F_{Y_t}\left(\psi^{\rightarrow}\left(x\right)\right) \right\vert dx\nonumber\\
		&\leq \int_{\psi\left(P_{0,t}\right)} \left\vert F_{X_t}\left(\psi^{\rightarrow}\left(x\right)\right)-F_{Y_t}\left(\psi^{\rightarrow}\left(x\right)\right) \right\vert dx\nonumber\\
		&=\int_{\psi\left(P_{0,t}\right) \setminus T_{\psi}} \left\vert F_{X_t}\left(\psi^{\rightarrow}\left(x\right)\right)-F_{Y_t}\left(\psi^{\rightarrow}\left(x\right)\right) \right\vert dx\nonumber\\
		&\leq \int_{P_{0,t}} \left\vert F_{X_t}\left(y\right)-F_{Y_t}\left(y\right) \right\vert d\psi\left(y\right)\nonumber\\
		&=\int_{P_{0,t} \medcap D_{\psi}} \left\vert F_{X_t}\left(y\right)-F_{Y_t}\left(y\right) \right\vert \psi'\left(y\right) dy\nonumber\\
		&\leq K \int_{P_{0,t} \medcap D_{\psi}} \left\vert F_{X_t}\left(y\right)-F_{Y_t}\left(y\right) \right\vert dy\nonumber\\
		&\leq K \int_{P_{0,t}} \left\vert F_{X_t}\left(y\right)-F_{Y_t}\left(y\right) \right\vert dy,
	\end{align}
	where the first inequality holds since $P_{0,t,\psi} \medcap C_{\psi} \subseteq \psi(P_{0,t})$, the second equality is due to the fact that $T_{\psi}$ is at most countable and hence $\int_{\psi\left(P_{0,t}\right) \medcap T_{\psi}} \left\vert F_{X_t}\left(\psi^{\rightarrow}\left(x\right)\right)-F_{Y_t}\left(\psi^{\rightarrow}\left(x\right)\right) \right\vert dx=0$, the second inequality follows since $x \in \psi(P_{0,t}) \setminus T_{\psi}$ implies that $\psi^{\rightarrow}(x) \in P_{0,t}$, the third equality holds since $D_{\psi}^c$ has Lebesgue measure $0$, by the hypothesis. The proof now follows from \eqref{1-eq-theorem-increasing-function-asymptotic-usual-stochastic-order-L-1-1-C-psi} and the hypothesis that $X_t \leq_{L_1\textnormal{-ast}} Y_t$, as $t \to \infty$.
\end{proof}

\begin{lemma}\label{1-lemma-L-1-ast-negation}
	Let $\left\{X_t : t \in T\right\}$ and $\left\{Y_t : t \in T\right\}$ be two stochastic processes such that $X_t \leq_{L_1\textnormal{-ast}} Y_t$, as $t \to \infty$. Then $-Y_t \leq_{L_1\textnormal{-ast}} -X_t$, as $t \to \infty$.
\end{lemma}

\begin{proof}
	Let us fix $t \in T$ and define $P_{0,t}'=\{x \in \mathbb{R} : F_{-Y_t}(x)<F_{-X_t}(x)\}$. Then
	\[
	P_{0,t}'=\left\{x \in \mathbb{R} : F_{X_t}(-x)-P\{X_t=-x\}<F_{Y_t}(-x)-P\{Y_t=-x\}\right\}.
	\]
	Let $M_t=\{x \in \mathbb{R} : \max{\{P\{X_t=-x\}, P\{Y_t=-x\}\}}>0\}$. Then $M_t$ is at most countable. Now, for every $t \in T$, we have
	\begin{align}
		&\,\,\,\,\,\,\,\int_{P_{0,t}'} \left\vert F_{-X_t}(x)-F_{-Y_t}(x) \right\vert dx\nonumber\\
		&=\int_{P_{0,t}'} \left\vert F_{Y_t}(-x)-F_{X_t}(-x)+P\{X_t=-x\}-P\{Y_t=-x\} \right\vert dx\nonumber\\
		&=\int_{P_{0,t}' \cap M_t} \left\vert P\{Y_t<-x\}-P\{X_t<-x\} \right\vert dx+\int_{P_{0,t}' \setminus M_t} \left\vert F_{Y_t}(-x)-F_{X_t}(-x) \right\vert dx\label{1-eq-proposition-L-1-ast-negation}\\
		&=\int_{P_{0,t}' \setminus M_t} \left\vert F_{Y_t}(-x)-F_{X_t}(-x) \right\vert dx\nonumber\\
		&\leq \int_{P_{0,t}} \left\vert F_{X_t}(x)-F_{Y_t}(x) \right\vert dx,\label{1-eq-proposition-L-1-ast-negation-2}
	\end{align}
	where the equality follows since the first integral in \eqref{1-eq-proposition-L-1-ast-negation} is $0$ and the last inequality is due to the fact that $P_{0,t}' \setminus M_t \subseteq -P_{0,t}:=\left\{-x : x \in P_{0,t}\right\}$. The proof follows since the integral in \eqref{1-eq-proposition-L-1-ast-negation-2} goes to $0$, as $t \to \infty$, by the hypothesis.
\end{proof}

The following corollary follows immediately from \cref{1-lemma-L-1-ast-negation} and \cref{1-theorem-increasing-function-asymptotic-usual-stochastic-order-L-1}.

\begin{corollary}\label{1-corollary-increasing-function-asymptotic-usual-stochastic-order-L-1}
	Let $\left\{X_t : t \in T\right\}$ and $\left\{Y_t : t \in T\right\}$ be two stochastic processes such that $X_t \leq_{L_1\textnormal{-ast}} Y_t$, as $t \to \infty$. Also, let $\psi:\mathbb{R} \to \mathbb{R}$ be a nonincreasing function, which is differentiable almost everywhere \textnormal{(}with respect to the Lebesgue measure\textnormal{)} with bounded derivative. Then $\psi(Y_t) \leq_{L_1\textnormal{-ast}} \psi(X_t)$, as $t \to \infty$. \hfill $\blacksquare$
\end{corollary}

The next lemma may be used in characterizing $L_1$-asymptotic usual stochastic order in terms of quantile function, given in \cref{1-theorem-W_1-L_1-2}. The proof is omitted.

\begin{lemma}\label{1-lemma-W_1-L_1-1}
	Let $F_X$ and $F_Y$ be any two cdfs. Write $P_0=\{x \in \mathbb{R} : F_X(x)<F_Y(x)\}$ and $A_0=\{u \in (0,1) : F_X^{\leftarrow}>F_Y^{\leftarrow}\}$. Then we have
	\begin{equation}\label{eq-L-1-quantile-characterization}
		\int_{P_0} \left\vert F_X(x)-F_Y(x) \right\vert dx=\int_{A_0} \left\vert F_X^{\leftarrow}(u)-F_Y^{\leftarrow}(u) \right\vert du. \tag*{$\blacksquare$}
	\end{equation}
\end{lemma}

The next theorem, which follows from \cref{1-definition-asymptotic-usual-stochastic-order-L-1} and \cref{1-lemma-W_1-L_1-1}, gives a characterization of the $L_1$-asymptotic usual stochastic order in terms of quantile functions.

\begin{theorem}\label{1-theorem-W_1-L_1-2}
	$X_t \leq_{L_1\textnormal{-ast}} Y_t$, as $t \to \infty$ if and only if
	\begin{equation}\label{1-eq-theorem-W_1-L_1-2}
		\lim_{t \to \infty} \int_{A_{0,t}} \left\vert F_{X_t}^{\leftarrow}(u)-F_{Y_t}^{\leftarrow}(u) \right\vert du=0,
	\end{equation}
	where $A_{0,t}=\{u \in (0,1) : F_{X_t}^{\leftarrow}(u)>F_{Y_t}^{\leftarrow}(u)\}$. \hfill $\blacksquare$
\end{theorem}

The following theorem shows that the $L_1$-asymptotic usual stochastic order implies the mean order in an asymptotic sense. We need the following lemma in order to prove it.

\begin{lemma}\label{1-lemma-ast-expectation-inequality-L-1}
	Let $X$ be a random variable with cdf $F$. Then
	\begin{equation}\label{1-eq-lemma-ast-expectation-inequality-L-1}
		E\left(X\right)=\int_0^1 F^{\leftarrow}\left(u\right)\,du.
	\end{equation}
\end{lemma}

\begin{theorem}\label{1-theorem-ast-expectation-inequality-L-1}
	Let $\left\{X_t : t \in T\right\}$ and $\left\{Y_t : t \in T\right\}$ be two stochastic processes such that $X_t \leq_{L_1\textnormal{-ast}} Y_t$, as $t \to \infty$. Then $\lim_{t \to \infty} E\left(X_t\right) \leq \lim_{t \to \infty} E\left(Y_t\right)$, if the limits exist.
\end{theorem}

\begin{proof}
	By \cref{1-lemma-ast-expectation-inequality-L-1}, we have
	\begin{align}
		E\left(X_t\right)&=\int_0^1 F_{X_t}^{\leftarrow}(u)\,du\nonumber\\
		&=\int_{A_{0,t}} F_{X_t}^{\leftarrow}(u)\,du+\int_{(0,1) \setminus A_{0,t}} F_{X_t}^{\leftarrow}(u)\,du\nonumber\\
		&\leq \int_{A_{0,t}} F_{X_t}^{\leftarrow}(u)\,du+\int_{(0,1) \setminus A_{0,t}} F_{Y_t}^{\leftarrow}(u)\,du\label{1-eq-theorem-ast-expectation-inequality-L-1}\\
		&=\int_{A_{0,t}} \left\vert F_{X_t}^{\leftarrow}(u)-F_{Y_t}^{\leftarrow}(u) \right\vert du+E\left(Y_t\right),\label{1-eq-theorem-ast-expectation-inequality-L-1-2}
	\end{align}
	where the inequality follows because $F_{X_t}^{\leftarrow}(u) \leq F_{Y_t}^{\leftarrow}(u)$ if and only if $u \in (0,1) \setminus A_{0,t}$. By taking limit as $t \to \infty$ in \eqref{1-eq-theorem-ast-expectation-inequality-L-1-2} the proof follows from the hypothesis and \cref{1-theorem-W_1-L_1-2}.
\end{proof}

\ni{\it Remark.} The conclusion of \cref{1-theorem-ast-expectation-inequality-L-1} cannot be reached if we assume the asymptotic usual stochastic order, discussed in \cref{1-section-ast}, instead of the $L_1$-asymptotic usual stochastic order. To see this, consider two stochastic processes $\{X_t : t \in T\}$ and $\{Y_t : t \in T\}$ such that $(i)$ $X_t \leq_{\textnormal{ast}} Y_t$, as $t \to \infty$, $(ii)$ $X_t \nleq_{L_1\textnormal{-ast}} Y_t$, as $t \to \infty$, $(iii)$ $F_{X_t}^{\leftarrow}(u)=F_{Y_t}^{\leftarrow}(u)$, whenever $u \notin A_{0,t}$ and $(iv)$ $\lim_{t \to \infty} E\left(X_t\right)$ and $\lim_{t \to \infty} E\left(Y_t\right)$ exist. Note that the four conditions are consistent, in the sense that no three of them contradict the other. Then, by $(iii)$, the inequality in \eqref{1-eq-theorem-ast-expectation-inequality-L-1} becomes an equality, i.e.
\begin{equation}\label{1-eq-theorem-ast-expectation-inequality-L-1-equality}
	E\left(X_t\right)=\int_{A_{0,t}} \left\vert F_{X_t}^{\leftarrow}(u)-F_{Y_t}^{\leftarrow}(u) \right\vert du+E\left(Y_t\right).
\end{equation}
By $(ii)$ and \cref{1-theorem-W_1-L_1-2}, we have $\lim_{t \to \infty} E\left(X_t\right)>\lim_{t \to \infty} E\left(Y_t\right)$, which contradicts the conclusion of \cref{1-theorem-ast-expectation-inequality-L-1}, even though $X_t \leq_{\textnormal{ast}} Y_t$, as $t \to \infty$.\vs

\ni{\it Remark.} Even when $\lim_{t \to \infty} E\left(X_t\right)$ and $\lim_{t \to \infty} E\left(Y_t\right)$ do not exist, then from \eqref{1-eq-theorem-ast-expectation-inequality-L-1}, we still have $\liminf_{t \to \infty} E\left(X_t\right) \leq \liminf_{t \to \infty} E\left(Y_t\right)$ if $\liminf_{t \to \infty} E\left(X_t\right)<\infty$ and $\limsup_{t \to \infty} E\left(X_t\right) \leq \limsup_{t \to \infty} E\left(Y_t\right)$ if $\limsup_{t \to \infty} E\left(X_t\right)<\infty$.\vs

\ni{\it Remark.} If $X_t \leq_{L_1\textnormal{-ast}} Y_t$, as $t \to \infty$, then it follows by combining \cref{1-theorem-increasing-function-asymptotic-usual-stochastic-order-L-1} and \cref{1-theorem-ast-expectation-inequality-L-1} that $\lim_{t \to \infty} E(\psi(X_t)) \leq \textnormal{(resp.} \geq\textnormal{)} \lim_{t \to \infty} E(\psi(Y_t))$, for every nondecreasing (resp. nonincreasing), continuous, almost everywhere differentiable function $\psi:\mathbb{R} \to \mathbb{R}$, with bounded derivative, when $\lim_{t \to \infty} E(\psi(X_t))$ exists. \hfill $\blacksquare$\vs

The following example shows that the counterpart of \cref{1-theorem-ast-convergence-in probability} does not hold for $L_1$-asymptotic usual stochastic order.

\begin{example}\label{1-counterexample-ast-convergence-in probability-L-1}
	Let $\{X_t : t \geq 0\}$ and $\{Y_t : t \geq 0\}$ be two stochastic processes with corresponding cdfs
	\begin{flalign*}
		&\hspace*{1cm} F_{X_t}(x)=
		\begin{cases*}
			\frac{1}{2(t+1)}e^{x+\frac{t+1}{2}}, & if  $-\infty < x \leq -\frac{t+1}{2}$,\\
			x+\frac{t+1}{2}+\frac{1}{2(t+1)}, & if $-\frac{t+1}{2} < x \leq -\frac{t+1}{2}+\frac{1}{2(t+1)}$,\\
			\left(1+\frac{t+1}{2}-\frac{1}{2(t+1)}\right)^{-1}\left(\frac{x-1}{t+1}\right)+\frac{2}{t+1}, & if  $-\frac{t+1}{2}+\frac{1}{2(t+1)} < x \leq 1$,\\
			(t-2)(x-1)+\frac{2}{t+1}, & if  $1 < x \leq 1+\frac{1}{t+1}$,\\
			1-\frac{1}{t+1}e^{-\left(x-1-\frac{1}{t+1}\right)}, & if  $1+\frac{1}{t+1} < x < \infty$;
		\end{cases*}\\
		&\textnormal{and} &\\
		&\hspace*{1cm} F_{Y_t}(x)=
		\begin{cases*}
			\frac{1}{2(t+1)}e^{x+t+1}, & if  $-\infty < x \leq -(t+1)$,\\
			x+\frac{1}{2(t+1)}+t+1, & if $-(t+1) < x \leq -(t+1)+\frac{1}{2(t+1)}$,\\
			\left(t+6-\frac{1}{2(t+1)}\right)^{-1}\left(\frac{x-5}{t+1}\right)+\frac{2}{t+1}, & if  $-(t+1)+\frac{1}{2(t+1)} < x \leq 5$,\\
			(t-2)(x-5)+\frac{2}{t+1}, & if  $5 < x \leq 5+\frac{1}{t+1}$,\\
			1-\frac{1}{t+1}e^{-\left(x-5-\frac{1}{t+1}\right)}, & if  $5+\frac{1}{t+1} < x < \infty$.
		\end{cases*}
	\end{flalign*}
	
	It is easy to verify that $X_t \overset{P}{\to} 1$ and $Y_t \overset{P}{\to} 5$. Observe that $F_{X_t}(x)<F_{Y_t}(x)$ whenever $-t-1+\frac{1}{2(t+1)} \leq x \leq -\frac{t+1}{2}$. Then it is easy to check that
	\begin{align*}
	\int_{P_{0,t}} \left\vert F_{X_t}(x)-F_{Y_t}(x) \right\vert dx &\geq \frac{1}{4}-\frac{1}{4(t+1)^2}\\
	&\to \frac{1}{4}, \textnormal{ as } t \to \infty.
	\end{align*}
	Thus we cannot have $X_t \leq_{L_1\textnormal{-ast}} Y_t$, as $t \to \infty$.
\end{example}\vs

\section{$\mathcal{W}_2$-asymptotic usual stochastic order}\label{1-section-asymptotic-W-2}

In this section, we define and analyze an asymptotic stochastic order, based on the $\mathcal{W}_2$ distance between distribution functions. It is stronger than the same based on the $L_1$ distance, in the sense that it implies the $L_1$-asymptotic usual stochastic order (see \cref{1-theorem-W_2-L_1}). Thus, if $X_t \leq_{L_1\textnormal{-ast}} Y_t$, as $t \to \infty$ and $W_t \leq_{L_1\textnormal{-ast}} Z_t$, as $t \to \infty$, then $\mathcal{W}_2$-asymptotic usual stochastic order may be useful to see the order between which pair of stochastic processes is stronger. We define the order as follows.

\begin{definition}\label{1-definition-asymptotic-usual-stochastic-order-W-2}
	Let $\left\{X_t : t \in T\right\}$ and $\left\{Y_t : t \in T\right\}$ be two stochastic processes, with respective classes of distribution functions $\left\{F_{X_t} : t \in T\right\}$ and $\left\{F_{Y_t} : t \in T\right\}$. We say that $X_t$ is smaller than $Y_t$ in $\mathcal{W}_2$-asymptotic usual stochastic order, denoted by $X_t \leq_{\mathcal{W}_2\textnormal{-ast}} Y_t$, as $t \to \infty$, if
	\begin{equation}\label{1-eq-definition-asymptotic-usual-stochastic-order-W-2}
		\lim_{t \to \infty} \int_{A_{0,t}} (F_{X_t}^{\leftarrow}(u)-F_{Y_t}^{\leftarrow}(u))^2\,du=0,
	\end{equation}
	where $A_{0,t}=\{u \in (0,1) : F_{X_t}^{\leftarrow}(u)>F_{Y_t}^{\leftarrow}(u)\}$. Also, we say that $X_t$ is equal to $Y_t$ in $\mathcal{W}_2$-asymptotic usual stochastic order and denote it by $X_t =_{\mathcal{W}_2\text{-ast}} Y_t$, as $t \to \infty$, if both $X_t \leq_{\mathcal{W}_2\textnormal{-ast}} Y_t$ and $Y_t \leq_{\mathcal{W}_2\textnormal{-ast}} X_t$ hold true, as $t \to \infty$.
\end{definition}

\ni{\it Remark.} In a similar spirit, we may define $\mathcal{W}_p$-asymptotic usual stochastic order, for every $p \geq 0$. Note that $p=0,1$ and $2$ respectively gives asymptotic usual stochastic order, $L_1$-asymptotic usual stochastic order and $\mathcal{W}_2$-asymptotic usual stochastic order. \hfill $\blacksquare$

\vspace*{0.25cm}
The next theorem shows that the $\mathcal{W}_2$-asymptotic usual stochastic order implies the $L_1$-asymptotic usual stochastic order. The proof follows from Cauchy-Schwartz inequality and \cref{1-theorem-W_1-L_1-2}.

\begin{theorem}\label{1-theorem-W_2-L_1}
	Let $\left\{X_t : t \in T\right\}$ and $\left\{Y_t : t \in T\right\}$ be two stochastic processes. If $X_t \leq_{\mathcal{W}_2\text{-ast}} Y_t$, as $t \to \infty$, then $X_t \leq_{L_1\textnormal{-ast}} Y_t$, as $t \to \infty$. \hfill $\blacksquare$
\end{theorem}

By combining \cref{1-theorem-ast-expectation-inequality-L-1} and \cref{1-theorem-W_2-L_1}, we obtain the next result, which shows that the $\mathcal{W}_2$-asymptotic usual stochastic order implies mean order in an asymptotic sense.

\begin{theorem}\label{1-theorem-ast-expectation-inequality-W-2}
	Let $\left\{X_t : t \in T\right\}$ and $\left\{Y_t : t \in T\right\}$ be two stochastic processes such that $X_t \leq_{\mathcal{W}_2\textnormal{-ast}} Y_t$, as $t \to \infty$. Then $\lim_{t \to \infty} E\left(X_t\right) \leq \lim_{t \to \infty} E\left(Y_t\right)$, if the limits exist.
\end{theorem}

\ni{\it Remark.} Even when $\lim_{t \to \infty} E\left(X_t\right)$ and $\lim_{t \to \infty} E\left(Y_t\right)$ do not exist, it is easy to see that $\liminf_{t \to \infty} E\left(X_t\right) \leq \liminf_{t \to \infty} E\left(Y_t\right)$ if $\liminf_{t \to \infty} E\left(X_t\right)<\infty$ and $\limsup_{t \to \infty} E\left(X_t\right) \leq \limsup_{t \to \infty} E\left(Y_t\right)$ if $\limsup_{t \to \infty} E\left(X_t\right)<\infty$. \hfill $\blacksquare$\vs

Next we discuss some properties of the $\mathcal{W}_2$-asymptotic usual stochastic order. The following proposition shows that equality of $X_t$ and $Y_t$ in $\mathcal{W}_2$-asymptotic usual stochastic order is characterized by convergence of their $\mathcal{W}_2$ distance to $0$, as $t \to \infty$. The proof is omitted.

\begin{proposition}\label{1-proposition-equality-asymptotic-usual-stochastic-order-W-2}
	$X_t =_{\mathcal{W}_2\textnormal{-ast}} Y_t$, as $t \to \infty$ if and only if $\Vert F_{X_t}-F_{Y_t} \Vert_{\mathcal{W}_2} \to 0$, as $t \to \infty$. \hfill $\blacksquare$
\end{proposition}

The next result shows that $\mathcal{W}_2$-asymptotic usual stochastic order is a partial order.

\begin{theorem}\label{1-theorem-transitivity-W-2-asymptotic-usual-stochastic-order}
	$\mathcal{W}_2$-asymptotic usual stochastic order is a partial order.
\end{theorem}

\begin{proof}
	$\mathcal{W}_2$-asymptotic usual stochastic order is trivially {\it reflexive} and {\it antisymmetric}. Let $\left\{X_t : t \in T\right\}$, $\left\{Y_t : t \in T\right\}$ and $\left\{Z_t : t \in T\right\}$ be three stochastic processes such that $X_t \leq_{\mathcal{W}_2\textnormal{-ast}} Y_t$, as $t \to \infty$ and $Y_t \leq_{\mathcal{W}_2\textnormal{-ast}} Z_t$, as $t \to \infty$. To prove {\it transitivity}, we need to show that $X_t \leq_{\mathcal{W}_2\textnormal{-ast}} Z_t$, as $t \to \infty$. Let $A_{0,t}=\left\{u \in (0,1) : F_{X_t}^{\leftarrow}(u)>F_{Y_t}^{\leftarrow}(u)\right\}$, $B_{0,t}=\left\{u \in (0,1) : F_{Y_t}^{\leftarrow}(u)>F_{Z_t}^{\leftarrow}(u)\right\}$ and $C_{0,t}=\left\{u \in (0,1) : F_{X_t}^{\leftarrow}(u)>F_{Z_t}^{\leftarrow}(u)\right\}$. Then it can be shown that $C_{0,t} \subseteq A_{0,t} \medcup B_{0,t}$. Now, we consider the partition $C_{0,t}=D_1 \medcup D_2 \medcup D_3$, where $D_1=A_{0,t} \medcap B_{0,t}^c \medcap C_{0,t}$, $D_2=A_{0,t}^c \medcap B_{0,t} \medcap C_{0,t}$ and $D_3=A_{0,t} \medcap B_{0,t} \medcap C_{0,t}$. Then
	\begin{equation}\label{1-eq-theorem-transitivity-asymptotic-usual-stochastic-order}
		\int_{C_{0,t}} (F_{X_t}^{\leftarrow}(u)-F_{Z_t}^{\leftarrow}(u))^2\,du = \sum_{i=1}^3 \int_{D_i} (F_{X_t}^{\leftarrow}(u)-F_{Z_t}^{\leftarrow}(u))^2\,du.
	\end{equation}
	Now it is easy to see that
	\begin{equation*}
		\int_{D_1} (F_{X_t}^{\leftarrow}(u)-F_{Z_t}^{\leftarrow}(u))^2\,du \leq \int_{A_{0,t}} (F_{X_t}^{\leftarrow}(u)-F_{Y_t}^{\leftarrow}(u))^2\,du.
	\end{equation*}
	and
	\begin{equation*}
		\int_{D_2} (F_{X_t}^{\leftarrow}(u)-F_{Z_t}^{\leftarrow}(u))^2\,du \leq \int_{B_{0,t}} (F_{Y_t}^{\leftarrow}(u)-F_{Z_t}^{\leftarrow}(u))^2\,du.
	\end{equation*}
	Again if $u \in D_3$, we have
	\[
	\int_{D_3} (F_{X_t}^{\leftarrow}(u)-F_{Z_t}^{\leftarrow}(u))^2\,du \leq 2\left\{\int_{A_{0,t}} (F_{X_t}^{\leftarrow}(u)-F_{Y_t}^{\leftarrow}(u))^2\,du+\int_{B_{0,t}} (F_{Y_t}^{\leftarrow}(u)-F_{Z_t}^{\leftarrow}(u))^2\,du\right\}.
	\]
	Applying these upper bounds in \eqref{1-eq-theorem-transitivity-asymptotic-usual-stochastic-order}, we have
	\[
	\int_{C_{0,t}} (F_{X_t}^{\leftarrow}(u)-F_{Z_t}^{\leftarrow}(u))^2\,du \leq 3\left\{\int_{A_{0,t}} (F_{X_t}^{\leftarrow}(u)-F_{Y_t}^{\leftarrow}(u))^2\,du+\int_{B_{0,t}} (F_{Y_t}^{\leftarrow}(u)-F_{Z_t}^{\leftarrow}(u))^2\,du\right\}.
	\]
	The proof then follows by the hypotheses.
\end{proof}

The next result shows that $\mathcal{W}_2$-asymptotic usual stochastic order is closed under taking nondecreasing transformations under certain condition. The following lemma is required to prove the result.

\begin{lemma}\label{1-lemma-increasing-function-asymptotic-usual-stochastic-order}
	Let $f$ be a nondecreasing, right-continuous function and let $g$ be a nondecreasing, left-continuous function. Then $(f \circ g^{\rightarrow})^{\leftarrow}(u)=g(f^{\leftarrow}(u))$, for every $u \in (0,1)$.
\end{lemma}

\begin{theorem}\label{1-theorem-increasing-function-asymptotic-usual-stochastic-order}
	Let $\left\{X_t : t \in T\right\}$ and $\left\{Y_t : t \in T\right\}$ be two stochastic processes such that $X_t \leq_{\mathcal{W}_2\textnormal{-ast}} Y_t$, as $t \to \infty$. Also, let $\psi:\mathbb{R} \to \mathbb{R}$ be a nondecreasing, Lipschitz continuous function. Then $\psi(X_t) \leq_{\mathcal{W}_2\textnormal{-ast}} \psi(Y_t)$, as $t \to \infty$.
\end{theorem}

\begin{proof}
	First, we consider the situation where $\psi$ is nondecreasing. By continuity of $\psi$, we have $F_{\psi(X_t)}(x)=F_{X_t}(\psi^{\rightarrow}(x))$ and $F_{\psi(Y_t)}(x)=F_{Y_t}(\psi^{\rightarrow}(x))$, for every $x \in \mathbb{R}$. By \cref{1-lemma-increasing-function-asymptotic-usual-stochastic-order}, we have $F_{\psi(X_t)}^{\leftarrow}(u)=\psi(F_{X_t}^{\leftarrow}(u))$ and $F_{\psi(Y_t)}^{\leftarrow}(u)=\psi(F_{Y_t}^{\leftarrow}(u))$, for every $u \in (0,1)$. Let $A_{0,t}=\{u \in (0,1) : F_{X_t}^{\leftarrow}(u)>F_{Y_t}^{\leftarrow}(u)\}$ and $A_{0,t,\psi}=\{u \in (0,1) : F_{\psi(X_t)}^{\leftarrow}(u)>F_{\psi(Y_t)}^{\leftarrow}(u)\}$. By nondecreasingness of $\psi$, we have $A_{0,t,\psi} \subseteq A_{0,t}$. Also, since $\psi$ is Lipshcitz continuous, there exists a positive real number $K$, such that
	\[
	\int_{A_{0,t,\psi}} (F_{\psi(X_t)}^{\leftarrow}(u)-F_{\psi(Y_t)}^{\leftarrow}(u))^2\,du \leq K^2 \int_{A_{0,t}} (F_{X_t}^{\leftarrow}(u)-F_{Y_t}^{\leftarrow}(u))^2\,du.
	\]
	From the assumption that $X_t \leq_{\mathcal{W}_2\textnormal{-ast}} Y_t$, as $t \to \infty$, it follows that $\psi(X_t) \leq_{\mathcal{W}_2\textnormal{-ast}} \psi(Y_t)$, as $t \to \infty$.
\end{proof}

The following lemma, which is needed to prove \cref{1-corollary-decreasing-function-asymptotic-usual-stochastic-order}, shows that $\mathcal{W}_2$-asymptotic usual stochastic order is reversed under taking negatives of the stochastic processes.

\begin{lemma}\label{1-lemma-W-2-ast-negation}
	Let $\left\{X_t : t \in T\right\}$ and $\left\{Y_t : t \in T\right\}$ be two stochastic processes such that $X_t \leq_{\mathcal{W}_2\textnormal{-ast}} Y_t$, as $t \to \infty$. Then $-Y_t \leq_{\mathcal{W}_2\textnormal{-ast}} -X_t$, as $t \to \infty$.
\end{lemma}

\begin{proof}
	Define $A_{0,t}'=\{u \in (0,1) : F_{-Y_t}^{\leftarrow}(u)>F_{-X_t}^{\leftarrow}(u)\}$, for every $t \in T$.
	It is easy to see that
	\[
	A_{0,t}'=\left\{u \in (0,1) : F_{X_t}^{\rightarrow}(1-u)>F_{Y_t}^{\rightarrow}(1-u)\right\}.
	\]
	Let $K_t=T_{F_{X_t}} \medcup T_{F_{Y_t}}$. Now, for every $t \in T$, we have
	\begin{align*}\label{1-eq-proposition-W-2-ast-negation}
		\int_{A_{0,t}'} \left\{ F_{-X_t}^{\leftarrow}(u)-F_{-Y_t}^{\leftarrow}(u) \right\}^2 du &=\int_{A_{0,t}' \setminus K_t} \left\{ F_{X_t}^{\leftarrow}(1-u)-F_{Y_t}^{\leftarrow}(1-u) \right\}^2 du\\
		&\leq \int_{A_{0,t}} \left\{ F_{X_t}^{\leftarrow}(u)-F_{Y_t}^{\leftarrow}(u) \right\}^2 du,
	\end{align*}
	where the equality follows since $K_t$ is countable and the inequality holds due to the fact that $u \in A_{0,t}' \setminus K_t$ implies $1-u \in A_{0,t}$. Hence the proof follows by hypothesis.
\end{proof}

Noting that $-\psi$ is nondecreasing if $\psi$ is nonincreasing, we immediately have the following corollary from \cref{1-lemma-W-2-ast-negation}.

\begin{corollary}\label{1-corollary-decreasing-function-asymptotic-usual-stochastic-order}
	Let $\left\{X_t : t \in T\right\}$ and $\left\{Y_t : t \in T\right\}$ be two stochastic processes such that $X_t \leq_{\mathcal{W}_2\textnormal{-ast}} Y_t$, as $t \to \infty$. Also, let $\psi:\mathbb{R} \to \mathbb{R}$ be a nonincreasing, Lipschitz continuous function. Then $\psi(Y_t) \leq_{\mathcal{W}_2\textnormal{-ast}} \psi(X_t)$, as $t \to \infty$. \hfill $\blacksquare$
\end{corollary}

The next corollary follows from \cref{1-theorem-ast-expectation-inequality-W-2} and \cref{1-theorem-increasing-function-asymptotic-usual-stochastic-order}.

\begin{corollary}
	If $X_t \leq_{\mathcal{W}_2\textnormal{-ast}} Y_t$, as $t \to \infty$, then $\lim_{t \to \infty} E(\psi(X_t)) \leq \lim_{t \to \infty} E(\psi(Y_t))$, for every nondecreasing, Lipschitz continuous function $\psi:\mathbb{R} \to \mathbb{R}$, when $\lim_{t \to \infty} E(\psi(X_t))$ exists.
\end{corollary}

\ni{\it Remark.} It follows from \cref{1-counterexample-ast-convergence-in probability-L-1} and \cref{1-theorem-W_2-L_1} that the assumptions of \cref{1-theorem-ast-convergence-in probability} are not sufficient to imply that $X_t \leq_{\mathcal{W}_2\textnormal{-ast}} Y_t$, as $t \to \infty$.\vs

\section{Applications}\label{1-section-applications}

For any right-continuous distortion function $\phi$ and a cdf $F$, $\phi(F)$ is again a cdf, often called a distorted cdf. Let $T \subseteq \mathbb{R}$ be unbounded above and let $\left\{X_t : t \in T\right\}$ and $\left\{Y_t : t \in T\right\}$ be two stochastic processes. Assume that there exists a class of continuous distortion functions $\left\{\phi_t : t \in T\right\}$ such that $F_{X_t}(x)=\phi_t(F_X(x))$ and $F_{Y_t}(x)=\phi_t(F_Y(x))$ for every $x \in \mathbb{R}$, where $F_X$ and $F_Y$ are two baseline distributions. Let $A_0=\{u \in (0,1) : F_X^{\leftarrow}(u)>F_Y^{\leftarrow}(u)\}$ and $A_{0,t}=\{u \in (0,1) : F_{X_t}^{\leftarrow}(u)>F_{Y_t}^{\leftarrow}(u)\}$, for every $t \in T$.

\begin{lemma}\label{1-lemma-phi-F-inverse-phi-right-continuous}
	Let $\phi$ be a distortion function and let $F$ be a cdf. If $\phi$ is left-continuous, then $(\phi \circ F)^{\rightarrow}(u)=F^{\rightarrow}(\phi^{\rightarrow}(u))$, for every $u \in (0,1)$ and if $\phi$ is right-continuous, then $(\phi \circ F)^{\leftarrow}(u)=F^{\leftarrow}(\phi^{\leftarrow}(u))$, for every $u \in (0,1)$. \hfill $\blacksquare$
\end{lemma}

The next lemma follows from \cref{1-lemma-phi-F-inverse-phi-right-continuous} and the observations that $\phi(\phi^{\leftarrow}(u))=u$, for every $u \in (0,1)$ if $\phi$ is continuous and $\phi^{\leftarrow}(\phi(u))=u$, for every $u \in (0,1)$, if $\phi$ is strictly increasing.

\begin{lemma}\label{1-lemma-phi-F-phi-G-general-phi-right-continuous-phi-continuous}
	Let $F$ and $G$ be two cdfs and let $\phi$ be a distortion function. Let us define $A_0=\{u \in (0,1) : F^{\leftarrow}(u)>G^{\leftarrow}(u)\}$ and $A_{0,\phi}=\{u \in (0,1) : (\phi \circ F)^{\leftarrow}(u)>(\phi \circ G)^{\leftarrow}(u)\}$. Then $A_{0,\phi} \subseteq \phi\left(A_0\right)$ if $\phi$ is continuous and $A_{0,\phi} \supseteq \phi\left(A_0\right)$ if $\phi$ is strictly increasing. \hfill $\blacksquare$
\end{lemma}

By \cref{1-lemma-phi-F-phi-G-general-phi-right-continuous-phi-continuous}, we have $\mu\left(A_{0,t}\right) \leq \mu\left(\phi_t\left(A_0\right)\right)$, for every $t \in T$. Thus, a sufficient condition for $X_t \leq_{\textnormal{ast}} Y_t$, as $t \to \infty$ to hold, is given by
\begin{equation}\label{1-eq-ast-sufficient-condition}
	\lim_{t \to \infty} \mu\left(\phi_t\left(A_0\right)\right)=0.
\end{equation}

Using \cref{1-lemma-phi-F-inverse-phi-right-continuous}, we have $F_{X_t}^{\leftarrow}(u)=F_X^{\leftarrow}(\phi_t^{\leftarrow}(u))$ and $F_{Y_t}^{\leftarrow}(u)=F_Y^{\leftarrow}(\phi_t^{\leftarrow}(u))$, for every $u \in (0,1)$ and $t \in T$. Also, if $u \in A_{0,t}$, then $\phi_t^{\leftarrow}(u) \in A_0$. Again, $\phi_t^{\leftarrow}(u)=v$ implies that $u=\phi_t(v)$, by continuity of $\phi_t$. Hence, for every $t \in T$, we have
\begin{equation*}
\int_{A_{0,t}} \left\vert F_{X_t}^{\leftarrow}(u)-F_{Y_t}^{\leftarrow}(u) \right\vert du \leq \int_{A_0} \left\vert F_X^{\leftarrow}(v)-F_Y^{\leftarrow}(v) \right\vert d\phi_t(v).
\end{equation*}
Assume that $\phi_t$ is differentiable for every $t \in T$. Then it is easy to see that
\begin{equation*}
\int_{A_{0,t}} \left\vert F_{X_t}^{\leftarrow}(u)-F_{Y_t}^{\leftarrow}(u) \right\vert du \leq \mathcal{W}_1\left(F_X,F_Y\right) \sup_{v \,\in A_0} \phi_t'(v).
\end{equation*}
In a similar spirit, it can be shown that
\begin{equation*}
	\int_{A_{0,t}} \left( F_{X_t}^{\leftarrow}(u)-F_{Y_t}^{\leftarrow}(u) \right)^2 du \leq \mathcal{W}_2^2\left(F_X,F_Y\right) \sup_{v \,\in A_0} \phi_t'(v).
\end{equation*}
Now, $\mathcal{W}_1\left(F_X,F_Y\right)<\infty$ if $X$ and $Y$ have finite first order moments. Thus, $X_t \leq_{L_1\textnormal{-ast}} Y_t$, as $t \to \infty$ holds, if $E\vert X \vert<\infty$, $E\vert Y \vert<\infty$ and
\begin{equation}\label{1-eq-L-1-sufficient-condition}
	\lim_{t \to \infty} \sup_{v \,\in A_0} \phi_t'(v)=0.
\end{equation}
Again, $\mathcal{W}_2\left(F_X,F_Y\right)<\infty$ if $X$ and $Y$ have finite second order moments. Hence, $X_t \leq_{\mathcal{W}_2\textnormal{-ast}} Y_t$, as $t \to \infty$ holds, if $E\vert X \vert^2<\infty$, $E\vert Y \vert^2<\infty$ and \eqref{1-eq-L-1-sufficient-condition} holds.

\subsection*{Mixtures of order statistics with possibly non-integral sample size}\label{1-subsection-mixture}
Based on Ferguson's Dirichlet process, \citet{S_1977} generalized the notion of order statistics to introduce an order statistic process, consisting of a continuum of fractional order statistics from any sample of possibly non-integral sample size. Let $\alpha>0$ and $0<r<\alpha+1$, where $r$ may depend on $\alpha$. We define $\lambda=\lim_{\alpha \to \infty} \frac{r}{\alpha+1}$. The $r$th order statistic from a sample of size $\alpha$ from a probability distribution with cdf $F_X$, denoted by $X_{r:\alpha}$, has the cdf $\phi_{r,\alpha}(F_X(x))$, where
\begin{equation}\label{1-eq-phi-r-alpha-non-integral}
\phi_{r,\alpha}(u)=\frac{1}{B(r,\alpha-r+1)} \int_0^u v^{r-1}(1-v)^{\alpha-r}\,dv,
\end{equation}
for every $u \in [0,1]$. Let $\beta_{r,\alpha-r+1}$ be a random variable that follows the Beta distribution with parameters $r$ and $\alpha-r+1$. Clearly, $\phi_{r,\alpha}(u)=P\{\beta_{r,\alpha-r+1} \leq u\}$, for every $u \in [0,1]$. Let us fix $\epsilon>0$. Then by Chebyshev's inequality, we have
\begin{equation}\label{1-equation-order-statistics-non-integral-sample-size}
P\left\{\left\vert \beta_{r,\alpha-r+1}-\frac{r}{\alpha+1} \right\vert \geq \frac{\epsilon}{2}\right\} \leq \frac{4r(\alpha-r+1)}{(\alpha+1)^2(\alpha+2)\epsilon^2} < \frac{4}{(\alpha+2)\epsilon^2}.
\end{equation}
Now, there exists $\alpha_0>0$, such that $\vert \frac{r}{\alpha+1}-\lambda \vert<\frac{\epsilon}{2}$, whenever $\alpha \geq \alpha_0$. Then, for every $\alpha \geq \alpha_0$, we have
\begin{equation}\label{1-equation-order-statistics-non-integral-sample-size-2}
P\left\{\left\vert \beta_{r,\alpha-r+1}-\lambda \right\vert \geq \epsilon\right\} \leq P\left\{\left\vert \beta_{r,\alpha-r+1}-\frac{r}{\alpha+1} \right\vert \geq \frac{\epsilon}{2}\right\}.
\end{equation}
By \eqref{1-equation-order-statistics-non-integral-sample-size}, the right hand side of \eqref{1-equation-order-statistics-non-integral-sample-size-2} converges to $0$, as $\alpha \to \infty$. Hence, it follows that $\lim_{\alpha \to \infty} P\left\{\left\vert \beta_{r,\alpha-r+1}-\lambda \right\vert \geq \epsilon\right\}=0$. Consequently, as $\alpha \to \infty$, $P\{\beta_{r,\alpha-r+1} \leq u\}$ goes to $0$ (resp. $1$) if $u<\lambda$ (resp. $u>\lambda$). Also, note that $\phi_{r,\alpha}(0)=0$ and $\phi_{r,\alpha}(1)=1$, for every $\alpha>0$. Hence,
\begin{equation}\label{1-eq-lim-phi-n-order-statistics-non-integral-sample-size}
\lim_{\alpha \to \infty} \phi_{r,\alpha}(u)=\begin{cases*}
0, & if  $u<\lambda$ or $u=0$; \\
1, & if $u>\lambda$ or $u=1$.
\end{cases*}
\end{equation}

If $0<\lambda<1$, the limit at the point $u=\lambda$ can be any real number in the interval $[0,1]$, depending on the direction, as well as the rate, at which $r/(\alpha+1)$ converges to $\lambda$, as $\alpha \to \infty$. We do not discuss it further, as we only require the limit at the points $u<\lambda$ and $u>\lambda$, for our purpose.\vs

\ni{\it Remark.} If $\alpha$ is a positive integer and $r \in \{1,2,\ldots,\alpha\}$, then \eqref{1-eq-phi-r-alpha-non-integral} can be written as
\[
\phi_{r,\alpha}(u)=\sum_{i=r}^\alpha \binom{\alpha}{i} u^i (1-u)^{\alpha-i},
\]
for every $u \in [0,1]$ and the cdf of the $r$th order statistic, $X_{r:\alpha}$, is given by $F_{X_{r:\alpha}}(x)=\phi_{r,\alpha}(F_X(x))$, for every $x \in \mathbb{R}$. \hfill $\blacksquare$\vs

The cdf of a mixture of finitely many distributions with cdf $F_1,F_2,\ldots,F_p$ is given by $F(x)=\sum_{i=1}^p \alpha_i F_i(x)$, where $\alpha_i \geq 0$ for every $i \in \{1,\ldots,p\}$ and $\sum_{i=1}^p \alpha_i=1$. In statistics, mixture distributions are useful in modeling data coming from a population consisting of different subpopulations. It also has extensive applications in cluster analysis (\citet{MB_1988}), text classification (\citet{JV_2002}), speech recognition (\citet{RQD_2000}), disease mapping (\citet{SB_1993}), meta analysis (\citet{LMR_2003}) etc. In reliability theory, the distribution function of the lifetime $T$ of a coherent system with $n$ iid components (with lifetimes $X_1,X_2,\ldots,X_n$) have the following representation as a mixture of the order statistics $X_{1:n},X_{2:n},\ldots,X_{n:n}$.
\[
F_T(x)=\sum_{i=1}^n s_i F_{X_{i:n}}(x),
\]
for every $x \in \mathbb{R}$, where $(s_1,s_2,\ldots,s_n)$ is the signature vector, which is a distribution-free characterization of the system's design (see \citet{S_2007}). In this section, we are interested in asymptotic stochastic comparison of finite mixtures of order statistics coming from two homogeneous samples.\vs

Let $F_X$ and $F_Y$ be two cdfs. Let $\alpha>0$ and $s \in \mathbb{N}$. Let us choose $r_1,r_2,\ldots,r_s$ (which may depend on $\alpha$) such that $0<r_1<r_2<\cdots<r_s<\alpha+1$. For every $i \in\{1,2,\ldots,s\}$, we define $\lambda_i=\lim_{\alpha \to \infty} \frac{r_i}{\alpha+1}$ and assume that $\lambda_1<\lambda_2<\cdots<\lambda_s$. Also, let $\boldsymbol{w}=(w_1,w_2,\ldots,w_s) \in (0,1)^s$ with $\sum_{i=1}^s w_i=1$. Let us define the mixtures of order statistics (characterized by $\boldsymbol{w}$ and $\boldsymbol{r}$) from the respective random samples (of size $\alpha$) from $F_X$ and $F_Y$ as
\[
F_{X_{\boldsymbol{w},\boldsymbol{r},\alpha}}(x)=\sum_{i=1}^s w_i F_{X_{r_i:\alpha}}(x),\,\, F_{Y_{\boldsymbol{w},\boldsymbol{r},\alpha}}(x)=\sum_{i=1}^s w_i F_{Y_{r_i:\alpha}}(x),
\]
for every $x \in \mathbb{R}$. We wish to examine asymptotic stochastic order between the random processes, generated from the random variables with respective distribution functions $F_{X_{\boldsymbol{w},\boldsymbol{r},\alpha}}$ and $F_{Y_{\boldsymbol{w},\boldsymbol{r},\alpha}}$, as $\alpha \to \infty$. Now, we have, for every $x \in \mathbb{R}$,
\[
F_{X_{\boldsymbol{w},\boldsymbol{r},\alpha}}(x)=\sum_{i=1}^s w_i \phi_{r_i,\alpha}(F_X(x)),
\]
where $\phi_{r_i,\alpha}$ is defined similarly as $\phi_{r,\alpha}$, as in \eqref{1-eq-phi-r-alpha-non-integral}. We can write $F_{X_{\boldsymbol{w},\boldsymbol{r},\alpha}}(x)=\phi_{\boldsymbol{w},\boldsymbol{r},\alpha}(F_X(x))$, where 
\begin{equation}\label{1-eq-mixture-phi-n-alpha-gamma-n}
\phi_{\boldsymbol{w},\boldsymbol{r},\alpha}(u)=\sum_{i=1}^s w_i\,\phi_{r_i,\alpha}(u),
\end{equation}
for every $u \in [0,1]$, where $\phi_{r_i,\alpha}$ is defined similar to $\phi_{r,\alpha}$. From \eqref{1-eq-lim-phi-n-order-statistics-non-integral-sample-size}, for every $i=1,2,\ldots,s$, we have
\begin{equation}\label{1-eq-lim-phi-n-order-statistics-r-i}
\lim_{\alpha \to \infty} \phi_{r_i,\alpha}(u)=\begin{cases*}
0, & if  $u<\lambda_i$ or $u=0$; \\
1, & if $u>\lambda_i$ or $u=1$.
\end{cases*}
\end{equation}
Then, from \eqref{1-eq-mixture-phi-n-alpha-gamma-n} and \eqref{1-eq-lim-phi-n-order-statistics-r-i}, we have
\begin{equation}\label{1-eq-lim-phi-n-order-statistics-mixture}
\lim_{\alpha \to \infty} \phi_{\boldsymbol{w},\boldsymbol{r},\alpha}(u)=\begin{cases*}
0, & if  $u<\lambda_1$ or $u=0$;\\
\sum_{j=1}^i w_j, & if $\lambda_i<u<\lambda_{i+1}$,\,\, $i=1,2,\ldots,s-1$;\\
1, & if $u>\lambda_s$ or $u=1$.
\end{cases*}
\end{equation}
If $0<\lambda_i<1$, for some $i \in \{1,2,\ldots,s\}$, the limit of $\phi_{\boldsymbol{w},\boldsymbol{r},\alpha}$ at the point $u=\lambda_i$ can be any real number between $\sum_{j=1}^{i-1} w_j$ and $\sum_{j=1}^i w_j$, depending on the direction, as well as the rate, at which $r_i/\alpha$ converges to $\lambda_i$, as $\alpha \to \infty$. We do not discuss it further, as the limits of $\phi_{\boldsymbol{w},\boldsymbol{r},\alpha}$ at the points $u \in \{\lambda_1,\lambda_2,\ldots,\lambda_s\} \cap (0,1)$ are not required for our purpose.\vs

Under the above setup, we have the following result, which provides sufficient conditions for asymptotic usual stochastic order between $X_{\boldsymbol{w},\boldsymbol{r},\alpha}$ and $Y_{\boldsymbol{w},\boldsymbol{r},\alpha}$, as $\alpha \to \infty$ to hold.

\begin{theorem}\label{1-theorem-ast-mixture}
	Let $X_{\boldsymbol{w},\boldsymbol{r},\alpha}$ and $Y_{\boldsymbol{w},\boldsymbol{r},\alpha}$ be two random variables with respective distribution functions $F_{X_{\boldsymbol{w},\boldsymbol{r},\alpha}}$ and $F_{Y_{\boldsymbol{w},\boldsymbol{r},\alpha}}$, for every $\alpha>0$. Assume that each $\lambda_i$ is separated from $A_0$ by a positive distance, i.e.
	\[
	\min_{i\,\in\,\{1,2,\ldots,s\}}{\textnormal{dist}\left(\lambda_i,A_0\right)}>0.
	\]
	Then $X_{\boldsymbol{w},\boldsymbol{r},\alpha} \leq_{\textnormal{ast}} Y_{\boldsymbol{w},\boldsymbol{r},\alpha}$, as $\alpha \to \infty$. Furthermore, $X_{\boldsymbol{w},\boldsymbol{r},\alpha} \leq_{\textnormal{asp}} Y_{\boldsymbol{w},\boldsymbol{r},\alpha}$, as $\alpha \to \infty$ if $X$ and $Y$ are independent.
\end{theorem}

\begin{proof}
	Let $\delta=\min_{i\,\in\,\{1,2,\ldots,s\}}{\textnormal{dist}\left(\lambda_i,A_0\right)}$. Now, there can be four cases. $(1)$ $\lambda_1=0$, $\lambda_s<1$, $(2)$ $\lambda_1=0$, $\lambda_s=1$, $(3)$ $\lambda_1>0$, $\lambda_s<1$ and $(4)$ $\lambda_1>0$, $\lambda_s=1$. We consider the first case. The other three can be similarly treated. Without loss of generality, we assume that $\delta<\min{\{\min{\{\lvert \lambda_i-\lambda_j \rvert : 1 \leq i<j \leq s\}}/2,1-\lambda_s\}}$. Then $A_0 \subseteq \{\medcup_{i=1}^{s-1} (\lambda_i+\delta,\lambda_{i+1}-\delta)\} \medcup (\lambda_s+\delta,1)$. Consequently, $\phi_{\boldsymbol{w},\boldsymbol{r},\alpha}(A_0) \subseteq \{\medcup_{i=1}^{s-1} (\phi_{\boldsymbol{w},\boldsymbol{r},\alpha}(\lambda_i+\delta),\phi_{\boldsymbol{w},\boldsymbol{r},\alpha}(\lambda_{i+1}-\delta))\} \medcup (\phi_{\boldsymbol{w},\boldsymbol{r},\alpha}(\lambda_s+\delta),1)$. Hence,
	\begin{equation*}
		\mu\left(\phi_{\boldsymbol{w},\boldsymbol{r},\alpha}\left(A_0\right)\right) \leq \sum_{i=1}^{s-1} \left\{\phi_{\boldsymbol{w},\boldsymbol{r},\alpha}(\lambda_{i+1}-\delta)-\phi_{\boldsymbol{w},\boldsymbol{r},\alpha}(\lambda_i+\delta)\right\}+1-\phi_{\boldsymbol{w},\boldsymbol{r},\alpha}(\lambda_s+\delta).
	\end{equation*}
	Taking limit as $\alpha \to \infty$ and applying \eqref{1-eq-lim-phi-n-order-statistics-mixture}, we obtain $\lim_{t \to \infty} \mu\left(\phi_{\boldsymbol{w},\boldsymbol{r},\alpha}\left(A_0\right)\right) \leq 0$, which implies \eqref{1-eq-ast-sufficient-condition}. Hence, $X_{\boldsymbol{w},\boldsymbol{r},\alpha} \leq_{\textnormal{ast}} Y_{\boldsymbol{w},\boldsymbol{r},\alpha}$, as $\alpha \to \infty$. Since, $F_X$ and $F_Y$ are independent, it follows from \cref{1-theorem-ast-asp} that $X_{\boldsymbol{w},\boldsymbol{r},\alpha} \leq_{\textnormal{asp}} Y_{\boldsymbol{w},\boldsymbol{r},\alpha}$, as $\alpha \to \infty$.
\end{proof}

\ni{\it Remark.} Let $s=1$. Then the situations where $\lambda=0$ or $\lambda=1$ cover all the extreme order statistics, whereas all the central order statistics fall under the case where $0<\lambda<1$. \hfill $\blacksquare$\vs

We immediately have the next corollary, which gives sufficient conditions for equality in asymptotic usual stochastic order between $X_{\boldsymbol{w},\boldsymbol{r},\alpha}$ and $Y_{\boldsymbol{w},\boldsymbol{r},\alpha}$, as $\alpha \to \infty$. Recall that $A_1=\{u \in (0,1) : F_X^{\leftarrow}(u)<F_Y^{\leftarrow}(u)\}$.

\begin{corollary}\label{1-corollary-ast-mixture-equality}
	Let $X_{\boldsymbol{w},\boldsymbol{r},\alpha}$ and $Y_{\boldsymbol{w},\boldsymbol{r},\alpha}$ be two random variables with respective distribution functions $F_{X_{\boldsymbol{w},\boldsymbol{r},\alpha}}$ and $F_{Y_{\boldsymbol{w},\boldsymbol{r},\alpha}}$, for every $\alpha>0$. Assume that each $\lambda_i$ is separated from $A_0 \medcup A_1$ by a positive distance, i.e.
	\[
	\min_{i\,\in\,\{1,2,\ldots,s\}}{\textnormal{dist}\left(\lambda_i,A_0 \medcup A_1\right)}>0.
	\]
	Then $X_{\boldsymbol{w},\boldsymbol{r},\alpha} =_{\textnormal{ast}} Y_{\boldsymbol{w},\boldsymbol{r},\alpha}$, as $n \to \infty$. Furthermore, $X_{\boldsymbol{w},\boldsymbol{r},\alpha} =_{\textnormal{asp}} Y_{\boldsymbol{w},\boldsymbol{r},\alpha}$, as $\alpha \to \infty$ if $F_X$ and $F_Y$ are independent. \hfill $\blacksquare$
\end{corollary}

The next result derives gradually weaker lower bounds for the quantity $\lim_{\alpha \to \infty} P\{X_{\boldsymbol{w},\boldsymbol{r},\alpha} \leq Y_{\boldsymbol{w},\boldsymbol{r},\alpha}\}$ under certain conditions.

\begin{theorem}\label{1-theorem-asp-mixture-limit-1}
	Let $X_{\boldsymbol{w},\boldsymbol{r},\alpha}$ and $Y_{\boldsymbol{w},\boldsymbol{r},\alpha}$ be two random variables with respective distribution functions $F_{X_{\boldsymbol{w},\boldsymbol{r},\alpha}}$ and $F_{Y_{\boldsymbol{w},\boldsymbol{r},\alpha}}$, for every $\alpha>0$. If $F_X$ and $F_Y$ are independent, continuous and $F_X^{\rightarrow}(\lambda_i)<F_Y^{\leftarrow}(\lambda_i)$, for every $i=1,2,\ldots,s$, then
	\begin{equation}\label{1-eq-theorem-asp-mixture-limit-1}
	\lim_{\alpha \to \infty} P\{X_{\boldsymbol{w},\boldsymbol{r},\alpha} \leq Y_{\boldsymbol{w},\boldsymbol{r},\alpha}\} \geq \max_{i \in \{1,2,\ldots,s\}} \left(w_i \sum_{j=1}^i w_j\right) \geq \max_{i \in \{1,2,\ldots,s\}} w_i^2 \geq \frac{1}{s^2}.
	\end{equation}
\end{theorem}

\begin{proof}
	Observe that
	\[
	P\{X_{\boldsymbol{w},\boldsymbol{r},\alpha} \leq Y_{\boldsymbol{w},\boldsymbol{r},\alpha}\}=\int_{-\infty}^{\infty} F_{X_{\boldsymbol{w},\boldsymbol{r},\alpha}}(x)\,dF_{Y_{\boldsymbol{w},\boldsymbol{r},\alpha}}(x).
	\]
	Again, note that, for the condition to hold, we must have $\lambda_1>0$ and $\lambda_s<1$, i.e., all the order statistics involved are central order statistics. Let $S=\{x \in \mathbb{R} : F_X(x) \geq F_Y(x)\}$. Also, let us fix $i \in \{1,2,\ldots,s\}$. Now, $F_X(x)>\lambda_i$, for every $x>F_X^{\rightarrow}(\lambda_i)$ and $F_Y(x)<\lambda_i$, for every $x<F_Y^{\leftarrow}(\lambda_i)$. In particular, $F_X(F_Y^{\leftarrow}(\lambda_i))>\lambda_i$ and $F_Y(F_X^{\rightarrow}(\lambda_i))<\lambda_i$. Define $l_{\lambda_i}=F_X^{\rightarrow} F_Y F_X^{\rightarrow}(\lambda_i)$ and $r_{\lambda_i}=F_Y^{\leftarrow} F_X F_Y^{\leftarrow}(\lambda_i)$. It is easy to check that $l_{\lambda_i}<F_X^{\rightarrow}(\lambda_i)$,
	$r_{\lambda_i}>F_Y^{\leftarrow}(\lambda_i)$ and $F_X(x) > F_Y(x)$, for every $x \in (l_{\lambda_i}, r_{\lambda_i})$. Clearly, $S \supseteq [l_{\lambda_i}, r_{\lambda_i}]$. Again, for every $i \in \{1,2,\ldots,s\}$, we can choose $\xi_i \in (F_X^{\rightarrow}(\lambda_i), \min{\{F_Y^{\leftarrow}(\lambda_i),F_X^{\rightarrow}(\lambda_{i+1})\}})$. Then
	\begin{align*}
	P\{X_{\boldsymbol{w},\boldsymbol{r},\alpha} \leq Y_{\boldsymbol{w},\boldsymbol{r},\alpha}\} &\geq \int_{S} F_{X_{\boldsymbol{w},\boldsymbol{r},\alpha}}(x)\,dF_{Y_{\boldsymbol{w},\boldsymbol{r},\alpha}}(x)\\
	&=\int_S \phi_{\boldsymbol{w},\boldsymbol{r},\alpha}(F_X(x))\,d\phi_{\boldsymbol{w},\boldsymbol{r},\alpha}(F_Y(x))\\
	&\geq \max_{i \in \{1,2,\ldots,s\}} \int_{l_{\lambda_i}}^{r_{\lambda_i}} \phi_{\boldsymbol{w},\boldsymbol{r},\alpha}(F_X(x))\,d\phi_{\boldsymbol{w},\boldsymbol{r},\alpha}(F_Y(x))\\
	&\geq \max_{i \in \{1,2,\ldots,s\}} \int_{\xi_i}^{r_{\lambda_i}} \phi_{\boldsymbol{w},\boldsymbol{r},\alpha}(F_X(x))\,d\phi_{\boldsymbol{w},\boldsymbol{r},\alpha}(F_Y(x))\\
	&\geq \max_{i \in \{1,2,\ldots,s\}} \phi_{\boldsymbol{w},\boldsymbol{r},\alpha}(F_X(\xi_i)) \int_{\xi_i}^{r_{\lambda_i}} d\phi_{\boldsymbol{w},\boldsymbol{r},\alpha}(F_Y(x))\\
	&=\max_{i \in \{1,2,\ldots,s\}} \phi_{\boldsymbol{w},\boldsymbol{r},\alpha}(F_X(\xi_i)) \{\phi_{\boldsymbol{w},\boldsymbol{r},\alpha}(F_Y(r_{\lambda_i}))-\phi_{\boldsymbol{w},\boldsymbol{r},\alpha}(F_Y(\xi_i))\}.
	\end{align*}
	For every $i \in \{1,2,\ldots,s\}$, we have $F_X(\xi_i)>\lambda_i$, which is greater than $F_Y(\xi_i)$. Finally, note that $F_Y(r_{\lambda_i})=F_X F_Y^{\leftarrow}(\lambda_i)>\lambda_i$. From these, we have by \eqref{1-eq-lim-phi-n-order-statistics-mixture}, $\lim_{\alpha \to \infty} \phi_{\boldsymbol{w},\boldsymbol{r},\alpha}(F_X(\xi_i)) \geq \sum_{j=1}^i w_j$, $\lim_{\alpha \to \infty} \phi_{\boldsymbol{w},\boldsymbol{r},\alpha}(F_Y(\xi_i)) \leq \sum_{j=1}^{i-1} w_j$ and $\lim_{\alpha \to \infty} \phi_{\boldsymbol{w},\boldsymbol{r},\alpha}(F_Y(r_{\lambda_i})) \geq \sum_{j=1}^i w_j$. Hence
	\[
	\lim_{\alpha \to \infty} \int_{S} F_{X_{\boldsymbol{w},\boldsymbol{r},\alpha}}(x)\,dF_{Y_{\boldsymbol{w},\boldsymbol{r},\alpha}}(x) \geq \max_{i \in \{1,2,\ldots,s\}} \left(w_i \sum_{j=1}^i w_j\right).
	\]
	This proves the first inequality in \eqref{1-eq-theorem-asp-mixture-limit-1}. The second inequality is due to nonnegativity of the weights and finally the third inequality follows from the fact that, under the condition $\sum_{i=1}^s w_i=1$, we have $\max{\{w_1,w_2,\ldots,w_s\}} \geq 1/s$.
\end{proof}

It is easy to see that the three lower bounds given in \eqref{1-eq-theorem-asp-mixture-limit-1}, are gradually weaker, yet gradually easier to compute. For example, if $s=2$, $w_1=1/4$ and $w_2=3/4$, then the first, the second and the third lower bounds are $3/4$, $9/16$ and $1/4$, respectively. For the case $s=1$, the weakest lower bound becomes $1$, which leads to the following corollary.

\begin{corollary}\label{1-corollary-asp-order-statistics-limit-1-non-integral-sample-size}
	Let $\alpha>0$ and $0<r<\alpha+1$, which  may depend on $\alpha$. Let $X_{r:\alpha}$ and $Y_{r:\alpha}$ be the $r$th order statistics from two random samples of non-integral size $\alpha$, from continuous distributions with cdf $F_X$ and $F_Y$, respectively. Also, let $\lambda=\lim_{\alpha \to \infty} \frac{r}{\alpha}$. If $F_X$ and $F_Y$ are independent and $F_X^{\rightarrow}(\lambda)<F_Y^{\leftarrow}(\lambda)$, then $\lim_{\alpha \to \infty} P\{X_{r:\alpha} \leq Y_{r:\alpha}\}=1$. \hfill $\blacksquare$
\end{corollary}

Note that $F_X^{\rightarrow}(\lambda)<F_Y^{\leftarrow}(\lambda)$ holds only if $0<\lambda<1$. When $\lambda=1$, one may think $F_X(x) > F_Y(x)$, for every $x \geq d$, for some $d \in \mathbb{R}$ will give $\lim_{\alpha \to \infty} P\{X_{r:\alpha} \leq Y_{r:\alpha}\}=1$. The next example, somewhat counterintuitively, shows that it is not true. Similar example can be constructed for the case when $\lambda=0$.

\begin{example}\label{1-example-limit-largest-order-statistic-stochastic-precedence-1}
	We consider two baseline cdfs $F_X$ and $F_Y$ having support $\mathbb{R}$. We have $F_{X_{n:n}}(x)=(F_X(x))^n$ and $F_{Y_{n:n}}(x)=(F_Y(x))^n$, for every $x \in \mathbb{R}$. Then
	\begin{align*}
	P\{X_{n:n} \leq Y_{n:n}\}&=\int_{-\infty}^{\infty} F_{X_{n:n}}(x)\,dF_{Y_{n:n}}(x)=n\int_{-\infty}^{\infty} (F_X(x))^n (F_Y(x))^{n-1} f_Y(x)\,dx\\
	&=n\int_{-\infty}^c (F_X(x))^n (F_Y(x))^{n-1} f_Y(x)\,dx+n\int_c^{\infty} (F_X(x))^n (F_Y(x))^{n-1} f_Y(x)\,dx.
	\end{align*}
	Since $F_X(c)<1$ and $F_Y(c)<1$, the first integral is bounded above by $\left\{F_X(c)F_Y(c)\right\}^n$, which goes to $0$, as $n \to \infty$. Hence
	\[
	\lim_{n \to \infty} P\{X_{n:n} \leq Y_{n:n}\}=\lim_{n \to \infty} n\int_c^{\infty} (F_X(x))^n (F_Y(x))^{n-1} f_Y(x) dx.
	\]
	Let $F_Y$ be any continuous cdf with support $\mathbb{R}$. We specify $F_X$ in the region $(c,\infty)$ by
	\[
	F_X(x)=\frac{1+F_Y(x)}{2}.
	\]
	Clearly, $F_X(x)>F_Y(x)$, for every $x>c$. In the region $(-\infty,c)$, we keep $F_X$ unspecified with the only condition that it is continuous, having a density, since how $F_X$ behaves in this region has no effect on $\lim_{n \to \infty} P(X_{n:n} \leq Y_{n:n})$. With this choice of $F_X$, we have, by routine calculations,
	\begin{align*}\label{1-eq-example-limit-largest-order-statistic-stochastic-precedence-1-routine-calculations}
	n\int_c^{\infty} (F_X(x))^n (F_Y(x))^{n-1} f_Y(x)\,dx&=\frac{n}{2^n}\int_{F_Y(c)}^1 \sum_{i=0}^n \binom{n}{i} u^{n+i-1}\,du\\
	&\leq \frac{1}{2^n}\sum_{i=0}^n \binom{n}{i} \frac{n}{n+i}\\
	&\leq \frac{1}{2^n}\sum_{i=0}^{\left[n/2\right]} \binom{n}{i}+\frac{1}{3 \cdot 2^{n-1}}\sum_{i=\left[n/2\right]+1}^n \binom{n}{i}\\
	&\leq \frac{5}{6}+O\left(\frac{1}{\sqrt{n}}\right).
	\end{align*}
	Hence, we have $\lim_{n \to \infty} P\{X_{n:n} \leq Y_{n:n}\} \leq 5/6$, which disproves the supposition. \hfill $\blacksquare$
\end{example}

The next theorem shows that, under the setup and condition of \cref{1-theorem-ast-mixture}, we have $X_{\boldsymbol{w},\boldsymbol{r},\alpha} \leq_{L_1\textnormal{-ast}} Y_{\boldsymbol{w},\boldsymbol{r},\alpha}$, as $\alpha \to \infty$, if $\max{\{E\vert X \vert,E\vert Y \vert\}}<\infty$ and $X_{\boldsymbol{w},\boldsymbol{r},\alpha} \leq_{\mathcal{W}_2\textnormal{-ast}} Y_{\boldsymbol{w},\boldsymbol{r},\alpha}$, as $\alpha \to \infty$, if $\max{\{E\left(X^2\right),E\left(Y^2\right)\}}<\infty$.

\begin{theorem}\label{1-theorem-L-1-mixture}
	 Let $X_{\boldsymbol{w},\boldsymbol{r},\alpha}$ and $Y_{\boldsymbol{w},\boldsymbol{r},\alpha}$ be two random variables with respective distribution functions $F_{X_{\boldsymbol{w},\boldsymbol{r},\alpha}}$ and $F_{Y_{\boldsymbol{w},\boldsymbol{r},\alpha}}$, for every $\alpha>0$. Assume that each $\lambda_i$ is separated from $A_0$ by a positive distance, i.e.
	 \[
	 \min_{i\,\in\,\{1,2,\ldots,s\}}{\textnormal{dist}\left(\lambda_i,A_0\right)}>0.
	 \]
	 Then $X_{\boldsymbol{w},\boldsymbol{r},\alpha} \leq_{L_1\textnormal{-ast}} Y_{\boldsymbol{w},\boldsymbol{r},\alpha}$, as $\alpha \to \infty$, if $\max{\{E\vert X \vert,E\vert Y \vert\}}<\infty$ and $X_{\boldsymbol{w},\boldsymbol{r},\alpha} \leq_{\mathcal{W}_2\textnormal{-ast}} Y_{\boldsymbol{w},\boldsymbol{r},\alpha}$, as $\alpha \to \infty$, if $\max{\{E\left(X^2\right),E\left(Y^2\right)\}}<\infty$.
\end{theorem}

\begin{proof}
	Observe that the proof follows if $\lim_{\alpha \to \infty} \sup_{u \,\in A_0} \phi_{\boldsymbol{w},\boldsymbol{r},\alpha}'(u)=0$. Again, from \eqref{1-eq-mixture-phi-n-alpha-gamma-n}, we have $\phi_{\boldsymbol{w},\boldsymbol{r},\alpha}'(u)=\sum_{i=1}^s w_i\,\phi_{r_i,\alpha}'(u)$, for every $u \in (0,1)$. Then
	\begin{equation}\label{1-eq-theorem-L-1-mixture-1}
		\lim_{\alpha \to \infty} \sup_{u \,\in A_0} \phi_{\boldsymbol{w},\boldsymbol{r},\alpha}'(u)=\sum_{i=1}^s w_i \lim_{\alpha \to \infty} \sup_{u \,\in A_0} \phi_{r_i,\alpha}'(u).
	\end{equation}
	Let $\epsilon_1=\min{\{\textnormal{dist}\left(\lambda_i,A_0\right) : i\,\in\,\{1,2,\ldots,s\}\}}$ and $\epsilon_2=\min{\{\vert \lambda_{i+1}-\lambda_i \vert : i \in \{1,2,\ldots,s-1\}\}}$. Now, there can be four cases. $(1)$ $\lambda_1=0$, $\lambda_s<1$, $(2)$ $\lambda_1=0$, $\lambda_s=1$, $(3)$ $\lambda_1>0$, $\lambda_s<1$ and $(4)$ $\lambda_1>0$, $\lambda_s=1$. We consider the first case. The other three can be similarly treated with suitable modifications. Let $\epsilon=\min{\{\epsilon_1,\epsilon_2,(1-\lambda_s)/2\}}$. Then, for every $i \in \{2,3,\ldots,s\}$, we have $\sup_{u \,\in A_0} \phi_{r_1,\alpha}'(u) \leq \sup_{u \,\in (\lambda_1+\epsilon,1)} \phi_{r_1,\alpha}'(u)$ and $\sup_{u \,\in A_0} \phi_{r_i,\alpha}'(u) \leq \sup_{u \,\in (0,\lambda_i-\epsilon) \cup (\lambda_i+\epsilon,1)} \phi_{r_i,\alpha}'(u)$. Again, from \eqref{1-eq-phi-r-alpha-non-integral}, we have $\phi_{r_i,\alpha}'(u)=u^{r_i-1}(1-u)^{\alpha-r_i}/B(r_i,\alpha-r_i+1)$, for every $u \in (0,1)$, for every $i \in \{1,2,\ldots,s\}$. It is easy to see that, for large enough $\alpha$, $\phi_{r_1,\alpha}'$ is decreasing in $u \in (\lambda_1+\epsilon/2,1)$. Also, for large enough $\alpha$, $\phi_{r_i,\alpha}'$ is increasing in $u \in (0,\lambda_i-\epsilon/2)$ and is decreasing in $u \in (\lambda_i+\epsilon/2,1)$, for every $i \in \{2,3,\ldots,s\}$. Consequently, $\sup_{u \,\in (\lambda_1+\epsilon,1)} \phi_{r_1,\alpha}'(u)=\phi_{r_1,\alpha}'(\lambda_1+\epsilon)$ and $\sup_{u \,\in (0,\lambda_i-\epsilon) \cup (\lambda_i+\epsilon,1)} \phi_{r_i,\alpha}'(u)=\max{\{\phi_{r_i,\alpha}'(\lambda_i-\epsilon),\phi_{r_i,\alpha}'(\lambda_i+\epsilon)\}}$, for every $i \in \{2,3,\ldots,s\}$. Combining these observations, from \eqref{1-eq-theorem-L-1-mixture-1}, we obtain
	\begin{equation*}
		 \lim_{\alpha \to \infty} \sup_{u \,\in A_0} \phi_{\boldsymbol{w},\boldsymbol{r},\alpha}'(u) \leq w_1 \lim_{\alpha \to \infty} \phi_{r_1,\alpha}'(\lambda_1+\epsilon)+\sum_{i=2}^s w_i \max{\left\{\lim_{\alpha \to \infty} \phi_{r_i,\alpha}'(\lambda_i-\epsilon),\lim_{\alpha \to \infty} \phi_{r_i,\alpha}'(\lambda_i+\epsilon)\right\}}.
	\end{equation*}
	Thus, it is enough to show that $\lim_{\alpha \to \infty} \phi_{r_1,\alpha}'(\lambda_1+\epsilon)=0$ and $\lim_{\alpha \to \infty} \phi_{r_i,\alpha}'(\lambda_i-\epsilon)=\lim_{\alpha \to \infty} \phi_{r_i,\alpha}'(\lambda_i+\epsilon)=0$, for every $i \in \{2,3,\ldots,s\}$. Let us fix $i \in \{2,3,\ldots,s\}$. Now, it follows from \eqref{1-eq-lim-phi-n-order-statistics-r-i} that $\lim_{\alpha \to \infty} \phi_{r_i,\alpha}(\lambda_i-\epsilon/2)=0$. Let $\delta>0$ be arbitrary. Then there exists $x>0$ such that whenever $\alpha>x$, we have $\phi_{r_i,\alpha}(\lambda_i-\epsilon/2)<\delta$ and $\phi_{r_i,\alpha}'$ is increasing in $(0,\lambda_i-\epsilon/2)$. Let us choose $\alpha>x$. Suppose, for the sake of contradiction, that $\phi_{r_i,\alpha}'(\lambda_i-\epsilon) \geq 2\delta/\epsilon$. Then $\phi_{r_i,\alpha}'(u)>2\delta/\epsilon$, for every $u \in (\lambda_i-\epsilon,\lambda_i-\epsilon/2)$. Now, by mean value theorem, $\phi_{r_i,\alpha}(\lambda_i-\epsilon/2)=\phi_{r_i,\alpha}(\lambda_i-\epsilon)+(\epsilon/2) \phi_{r_i,\alpha}'(c)$, for some $c \in (\lambda_i-\epsilon,\lambda_i-\epsilon/2)$. Since, $\phi_{r_i,\alpha}'(c)>2\delta/\epsilon$, we have $\phi_{r_i,\alpha}(\lambda_i-\epsilon/2)>\delta$, which is a contradiction. Thus, we must have $\phi_{r_i,\alpha}'(\lambda_i-\epsilon)<2\delta/\epsilon$. Since, $\delta>0$ can be arbitrarily small, we obtain $\lim_{\alpha \to \infty} \phi_{r_i,\alpha}'(\lambda_i-\epsilon)=0$. A similar argument can show that $\lim_{\alpha \to \infty} \phi_{r_i,\alpha}'(\lambda_i+\epsilon)=0$. Hence the proof follows.
\end{proof}

The following corollary provides sufficient conditions for equality in $L_1$-asymptotic usual stochastic order of $X_{\boldsymbol{w},\boldsymbol{r},\alpha}$ and $Y_{\boldsymbol{w},\boldsymbol{r},\alpha}$, as $\alpha \to \infty$.

\begin{corollary}\label{1-corollary-L-1-mixture-equality}
	Let $X_{\boldsymbol{w},\boldsymbol{r},\alpha}$ and $Y_{\boldsymbol{w},\boldsymbol{r},\alpha}$ be two random variables with respective distribution functions $F_{X_{\boldsymbol{w},\boldsymbol{r},\alpha}}$ and $F_{Y_{\boldsymbol{w},\boldsymbol{r},\alpha}}$, for every $\alpha>0$. Assume that each $\lambda_i$ is separated from $A_0 \medcup A_1$ by a positive distance, i.e.
	\[
	\min_{i\,\in\,\{1,2,\ldots,s\}}{\textnormal{dist}\left(\lambda_i,A_0 \medcup A_1\right)}>0.
	\]
	 Then $X_{\boldsymbol{w},\boldsymbol{r},\alpha} =_{L_1\textnormal{-ast}} Y_{\boldsymbol{w},\boldsymbol{r},\alpha}$, as $\alpha \to \infty$, if $\max{\{E\vert X \vert,E\vert Y \vert\}}<\infty$ and $X_{\boldsymbol{w},\boldsymbol{r},\alpha} =_{\mathcal{W}_2\textnormal{-ast}} Y_{\boldsymbol{w},\boldsymbol{r},\alpha}$, as $\alpha \to \infty$, if $\max{\{E\left(X^2\right),E\left(Y^2\right)\}}<\infty$. \hfill $\blacksquare$
\end{corollary}

\subsection*{Record values}\label{1-subsection-record} Motivated by reports on extreme weather conditions, \citet{C_1952} introduced the notion of record values. Let $\left\{X_n : n\in \mathbb{N}\right\}$ be a sequence of random variables, generated independently from a distribution $F_X$. The first observation $X_1$ (to be precise, the realized value of $X_1$) is always a record. The next (upper) record occurs when for some $N \in \mathbb{N}$, $X_N>X_1$ but $X_n \leq X_1$ for every $n \leq N-1$ and so on. Following this idea, the sequence $\{T_n : n \in \mathbb{N}\}$ of {\it upper record times} is recursively defined by $T_1=1$ and $T_{n+1}=\min{\left\{i \in \mathbb{N} : X_i>X_{T_n}\right\}}$ for every $n \in \mathbb{N}$. The corresponding sequence of {\it upper record values} $\left\{R_n : n \in \mathbb{N}\right\}$ is defined by $R_n=X_{T_n}$ for every $n \in \mathbb{N}$. A detailed summary of record statistics is given in \citet{AN_2011}. \citet{DK_1976} generalized this concept to introduce $k$th record values, which is useful when one is interested in the records in terms of $k$th highest value, where $k$ is a positive integer. Here, the sequence $\{T_n^{(k)} : n \in \mathbb{N}\}$ of {\it upper $k$th record times} is defined by
\[
T_1^{(k)}=1 \text{ and } T_{n+1}^{(k)}=\min{\left\{i \in \mathbb{N} : X_{i:i+k-1}>X_{T_n^{(k)}:T_n^{(k)}+k-1}\right\}}, \text{ for every } n \in \mathbb{N}.
\]
The corresponding sequence of {\it upper $k$th record values} $\left\{R_n^{(k)} : n \in \mathbb{N}\right\}$ is defined by
\[
R_n^{(k)}=X_{T_n^{(k)}},
\]
for every $n \in \mathbb{N}$. Note that the upper $k$th record time and the upper $k$th record value reduce to their respective ordinary counterparts when $k=1$. Similarly, one can define the lower $k$th record time and the lower $k$th record value. In this work, we shall focus on upper $k$th record value. Analogous result for lower $k$th record value can be obtained in a similar way. From Corollary $2$ of \citet{DK_1976}, it follows that the distribution function of $R_n^{(k)}$ can be written as $\phi_{n,k}(F_X(x))$, where \begin{equation}\label{1-eq-kth-record-value-phi-n}
\phi_{n,k}(u)=\frac{1}{(n-1)!}\int_0^{-k\ln{(1-u)}} y^{n-1} e^{-y} dy,
\end{equation}
for every $u \in [0,1]$. Let $G_n$ be a random variable following gamma distribution with shape parameter $n$ and rate parameter $1$. We can write $G_n=\sum_{i=1}^n X_i$, where $X_i$ has exponential distribution with rate parameter $1$, for every $i \in \{1,2,\ldots,n\}$. Then, for every $u \in [0,1]$, we have
\[
\phi_{n,k}(u)=P\{G_n \leq -k\ln{(1-u)}\}=P\left\{\frac{\sum_{i=1}^n X_i-n}{\sqrt{n}} \leq -\frac{k\ln{(1-u)}}{\sqrt{n}}-\sqrt{n}\right\}.
\]
By central limit theorem, for every $u \in [0,1)$, we have
\begin{equation}\label{1-eq-lim-phi-n-record-values}
	\lim_{n \to \infty} \phi_{n,k}(u)=\lim_{n \to \infty} \Phi\left(-\frac{k\ln{(1-u)}}{\sqrt{n}}-\sqrt{n}\right)=\begin{cases*}
		0, & if  $u<1$; \\
		1, & if $u=1$.
	\end{cases*}
\end{equation}

The next two results provide sufficient conditions for asymptotic stochastic dominance and equality, from the perspectives of asymptotic usual stochastic order and asymptotic stochastic precedence order, for the $k$th record values from two different homogeneous samples of size $n$, as $n \to \infty$.

\begin{theorem}\label{1-theorem-ast-record-values}
	Let $F_X$ and $F_Y$ be two cdfs such that there exists $d \in \mathbb{R}$ satisfying $F_Y(d)<1$ and $F_X(x) \geq F_Y(x)$ for every $x \geq d$. Let $\left\{X_n : n \in \mathbb{N}\right\}$ and $\left\{Y_n : n \in \mathbb{N}\right\}$ be two sequences of iid random variables with respective parent distribution functions $F_X$ and $F_Y$ and respective sequences of $k$th record values $\{R_n^{(k)} : n \in \mathbb{N}\}$ and $\{S_n^{(k)} : n \in \mathbb{N}\}$. Then $R_n^{(k)} \leq_{\textnormal{ast}} S_n^{(k)}$, as $n \to \infty$, for every $k \in \mathbb{N}$. Furthermore, $R_n^{(k)} \leq_{\textnormal{asp}} S_n^{(k)}$, as $n \to \infty$, for every $k \in \mathbb{N}$, if $F_X$ and $F_Y$ are independent.
\end{theorem}

\begin{proof}
	Let us choose $u \in (F_Y(d),1)$ arbitrarily. By assumption, $F_Y(x) \geq u$ implies that $F_X(x) \geq u$, for every $x>d$. Thus, $\{x>d : F_Y(x) \geq u\} \subseteq \{x>d : F_X(x) \geq u\}$. Also, note that $F_Y(x)<u$, for every $x \leq d$. Then $F_Y^{\leftarrow}(u)=\inf{\{x \in \mathbb{R} : F_Y(x) \geq u\}}=\inf{\{x>d : F_Y(x) \geq u\}} \geq \inf{\{x>d : F_X(x) \geq u\}} \geq \inf{\{x \in \mathbb{R} : F_X(x) \geq u\}}=F_X^{\leftarrow}(u)$. Hence, $A_0 \subseteq (0,F_Y(d)]$. Consequently, $\mu\left(\phi_{k,n}\left(A_0\right)\right) \leq \phi_{k,n}(F_Y(d))$, which goes to $0$, as $n \to \infty$, by the assumption that $F_Y(d)<1$ and \eqref{1-eq-lim-phi-n-record-values}. The proof for asymptotic usual stochastic order then follows from the sufficient condition, given in \eqref{1-eq-ast-sufficient-condition}. If $F_X$ and $F_Y$ are independent, then the asymptotic stochastic precedence order holds by \cref{1-theorem-ast-asp}.
\end{proof}

The next corollary follows from the above theorem by observing that $A_0 \medcup A_1 \subseteq (0,F_Y(d)]$ under its hypothesis.

\begin{corollary}\label{1-corollary-ast-record-values}
	Let $F_X$ and $F_Y$ be two cdfs such that there exists $d \in \mathbb{R}$ satisfying $F_Y(d)<1$ and $F_X(x)=F_Y(x)$ for every $x \geq d$. Let $\left\{X_n : n \in \mathbb{N}\right\}$ and $\left\{Y_n : n \in \mathbb{N}\right\}$ be two sequences of iid random variables with respective parent distribution functions $F_X$ and $F_Y$ and respective sequences of $k$th record values $\{R_n^{(k)} : n \in \mathbb{N}\}$ and $\{S_n^{(k)} : n \in \mathbb{N}\}$. Then $R_n^{(k)} =_{\textnormal{ast}} S_n^{(k)}$, as $n \to \infty$, for every $k \in \mathbb{N}$. Furthermore, $R_n^{(k)} =_{\textnormal{asp}} S_n^{(k)}$, as $n \to \infty$, for every $k \in \mathbb{N}$, if $F_X$ and $F_Y$ are independent. \hfill $\blacksquare$
\end{corollary}

The next two results provide sufficient conditions for asymptotic stochastic dominance and equality, from the perspectives of $L_1$-asymptotic usual stochastic order and $\mathcal{W}_2$-asymptotic usual stochastic order, for $R_n^{(k)}$ and $S_n^{(k)}$, as $n \to \infty$.

\begin{theorem}\label{1-theorem-L-1-record-values}
	Let $F_X$ and $F_Y$ be two cdfs such that there exists $d \in \mathbb{R}$ satisfying $F_Y(d)<1$ and $F_X(x) \geq F_Y(x)$ for every $x \geq d$. Let $\left\{X_n : n \in \mathbb{N}\right\}$ and $\left\{Y_n : n \in \mathbb{N}\right\}$ be two sequences of iid random variables with respective parent distribution functions $F_X$ and $F_Y$ and respective sequences of $k$th record values $\{R_n^{(k)} : n \in \mathbb{N}\}$ and $\{S_n^{(k)} : n \in \mathbb{N}\}$. Then $R_n^{(k)} \leq_{L_1\textnormal{-ast}} S_n^{(k)}$, as $n \to \infty$, for every $k \in \mathbb{N}$ if $\max{\{E\vert X \vert,E\vert Y \vert\}}<\infty$ and $R_n^{(k)} \leq_{\mathcal{W}_2\textnormal{-ast}} S_n^{(k)}$, as $n \to \infty$, for every $k \in \mathbb{N}$ if $\max{\{E\left(X^2\right),E\left(Y^2\right)\}}<\infty$.
\end{theorem}

\begin{proof}
	Observe that the proof follows if $\lim_{n \to \infty} \sup_{u \,\in A_0} \phi_{k,n}'(u)=0$. From \eqref{1-eq-kth-record-value-phi-n}, simple calculation shows that $\phi_{k,n}'$ is increasing in $u \in (0,(F_Y(d)+1)/2)$, whenever $n>1-(k-1)\ln{(1-\{(F_Y(d)+1)/2\})}$. Since $A_0 \subseteq (0,F_Y(d)]$, we have $\sup_{u \,\in A_0} \phi_{k,n}'(u) \leq \phi_{k,n}'(F_Y(d))=k(1-F_Y(d))^{k-1}\{-k\ln{(1-F_Y(d))}\}^{n-1}/(n-1)!$, which goes to $0$, as $n \to \infty$.
\end{proof}

The proof of the next corollary follows from the above theorem by noting that $A_0 \cup A_1 \subseteq (0,F_Y(d)]$ under its hypothesis.

\begin{corollary}\label{1-corollary-L-1-record-values-equality}
	Let $F_X$ and $F_Y$ be two cdfs such that there exists $d \in \mathbb{R}$ satisfying $F_Y(d)<1$ and $F_X(x)=F_Y(x)$ for every $x \geq d$. Let $\left\{X_n : n \in \mathbb{N}\right\}$ and $\left\{Y_n : n \in \mathbb{N}\right\}$ be two sequences of iid random variables with respective parent distribution functions $F_X$ and $F_Y$ and respective sequences of $k$th record values $\{R_n^{(k)} : n \in \mathbb{N}\}$ and $\{S_n^{(k)} : n \in \mathbb{N}\}$. Then $R_n^{(k)} =_{L_1\textnormal{-ast}} S_n^{(k)}$, as $n \to \infty$, for every $k \in \mathbb{N}$ if $\max{\{E\vert X \vert,E\vert Y \vert\}}<\infty$ and $R_n^{(k)} =_{\mathcal{W}_2\textnormal{-ast}} S_n^{(k)}$, as $n \to \infty$, for every $k \in \mathbb{N}$ if $\max{\{E\left(X^2\right),E\left(Y^2\right)\}}<\infty$.
\end{corollary}

\section{Conclusion}\label{1-section-conclusion}

In this work, four different asymptotic stochastic orders have been proposed and analyzed. The problem of choosing one of these in a given situation can be tricky for a practitioner. The asymptotic stochastic orders developed in the quantitative approach have no direct connection with the ones defined in the semi-quantitative approach. See \cref{1-example-ast-L-1-1} and \cref{1-example-ast-L-1-2} for illustrations of this point. Keeping the fact that the chosen order for a given situation should give a conclusion that matches with that obtained from common sense in mind, we roughly state a general rule of thumb. If the sign, as well as the magnitude of the differences of the types $F_{X_t}^{\leftarrow}(u)-F_{Y_t}^{\leftarrow}(u)$ or $F_{X_t}(x)-F_{Y_t}(x)$ are important in a particular situation, the practitioners should apply an asymptotic stochastic order, derived from the quantitative approach. If the sign of these differences is important, but the magnitude of the same has no impact, then one may opt to choose the asymptotic usual stochastic order. Finally, if there is a strong indication that the stochastic processes are dependent, then any asymptotic stochastic order based on the usual stochastic order may give misleading conclusions and hence, the practitioners may choose to employ the asymptotic stochastic precedence order. Just like the different stochastic orders for random variables, the choice of the asymptotic stochastic order to be considered in a given situation is crucial and may result into misleading conclusions, if applied without proper judgment of the situation at hand. For instance, \cref{1-example-ast-L-1} depicts a situation where the asymptotic usual stochastic order does not work and one has to take a quantitative route to reach a justifiable conclusion. In the end, the orders considered here are only tools that compare stochastic processes asymptotically from specific perspectives. These are meant to guide the practitioner to reach a reasonable conclusion and should not be taken as final deciders.\vs

{\bf Acknowledgement.} The first author is supported by Senior Research Fellowship from University Grants Commission, Government of India.

\baselineskip=16pt

\bibliographystyle{newapa}
\bibliography{references}\vs

\section{Appendix}\label{1-section-appendix}

\begin{proof}[Proof of \cref{1-lemma-quantile-distribution}]
	Let $S=\{x \in \mathbb{R} : F_X(x) < F_Y(x)\}$. Now,
	\begin{equation}\label{1-proof-theorem-ast-inverse}
		P_X(S)=\int_S dF_X(x)=\mu\left(F_X(S)\right),
	\end{equation}
	where $F_X(S)=\{F_X(x) : x \in S\}$. If $u \in F_X(S)$, then there exists $x \in [F_X^{\leftarrow}(u),F_X^{\rightarrow}(u)]$, such that $F_Y(x) > u$. Note that, $F_Y(y) \leq u$, for every $y \leq F_Y^{\rightarrow}(u)$. Hence, $x > F_Y^{\rightarrow}(u)$. Consequently, $F_X^{\rightarrow}(u) > F_Y^{\rightarrow}(u)$. Thus, $u \in F_X(S) \Rightarrow F_X^{\rightarrow}(u) > F_Y^{\rightarrow}(u)$. Again, suppose that $u \notin F_X(S)$. Then, for every $x \in [F_X^{\leftarrow}(u),F_X^{\rightarrow}(u)]$, we must have $F_Y(x) \leq u$ and hence $F_Y^{\rightarrow}(u) \geq F_X^{\rightarrow}(u)$. Taking contrapositive, we have $F_X^{\rightarrow}(u) > F_Y^{\leftarrow}(u) \Rightarrow u \in F_X(S)$. Combining the observations, we obtain
	\begin{equation}\label{1-proof-theorem-ast-inverse-sets}
		F_X(S)=\left\{u \in (0,1) : F_X^{\rightarrow}(u) > F_Y^{\rightarrow}(u)\right\}.
	\end{equation}
	Substituting \eqref{1-proof-theorem-ast-inverse-sets} into \eqref{1-proof-theorem-ast-inverse}, the proof follows.
\end{proof}

\begin{proof}[Proof of \cref{1-lemma-no-inconsistency}]
	Since $\{x \in \mathbb{R} : F_X(x) < F_Y(x)\}$ is open, we can write the set as $\bigcup_{k=1}^{\infty} I_k$, where $\left\{I_k : k \in \mathbb{N}\right\}$ is a countable collection of disjoint open intervals. Then
	\begin{equation}\label{1-eq-no-inconsistency}
		P_X\left(\left\{x \in \mathbb{R} : F_{X_t}(x) < F_{Y_t}(x)\right\}\right)=\sum_{k=1}^{\infty} P_X\left(I_k\right)=\sum_{k=1}^{\infty} \left\{F_X\left(\sup{I_k}\right)-F_X\left(\inf{I_k}\right)\right\},
	\end{equation}
	where the last equality follows from continuity of $F_X$. Let $k \in \mathbb{N}$. By construction of $I_k$, there exists $\epsilon>0$ such that $F_X(x) > F_Y(x)$ if $x \in (\inf{I_k}-\epsilon,\inf{I_k}]$ and $F_X(x) \leq F_Y(x)$ if $x \in (\inf{I_k},\inf{I_k}+\epsilon)$. Then it follows by continuity of $F_X$ and $F_Y$ that $F_X\left(\inf{I_k}\right)=F_Y\left(\inf{I_k}\right)$. Similarly, we have $F_X\left(\sup{I_k}\right)=F_Y\left(\sup{I_k}\right)$. These equalities hold for every $k \in \mathbb{N}$. Then, from \eqref{1-eq-no-inconsistency}, we obtain that $P_X\left(\left\{x \in \mathbb{R} : F_X(x) < F_Y(x)\right\}\right)$ is equal to
	\[
	\sum_{k=1}^{\infty} \left\{F_Y\left(\sup{I_k}\right)-F_Y\left(\inf{I_k}\right)\right\}=\sum_{k=1}^{\infty} P_Y\left(I_k\right)=P_Y\left(\left\{x \in \mathbb{R} : F_X(x) < F_Y(x)\right\}\right),
	\]
	where the first equality is due to continuity of $F_Y$. Hence the proof.
\end{proof}

\begin{proof}[Proof of \cref{1-lemma-ast-asp}]
	Suppose that $F(x)<G(x)$. By right-continuity of $F$, we have $F(x)<F(F^{\leftarrow}(G(x)))$. Consequently, we obtain $x<F^{\leftarrow}(G(x))$. Also, observe that $G^{\leftarrow}(G(x)) \leq x$. Hence, $G(x) \in \{u \in (0,1) : F^{\leftarrow}(u)>G^{\leftarrow}(u)\}$, which proves \eqref{1-eq-lemma-ast-asp}. Now, suppose that $G$ is continuous. Let $F^{\leftarrow}(u)>G^{\leftarrow}(u)$. Then we have $F(G^{\leftarrow}(u))<u=G(G^{\leftarrow}(u))$. Hence $G^{\leftarrow}(u) \in \{x \in \mathbb{R} : F(x)<G(x)\}$. Thus, $G(G^{\leftarrow}(u))=u \in \{G(x) : F(x)<G(x)\}$, which proves the reverse direction of \eqref{1-eq-lemma-ast-asp}.
\end{proof}

\begin{proof}[Proof of \cref{1-lemma-ast-expectation-inequality-L-1}]
	Let $U$ be a random variable following uniform distribution with support $[0,1]$ and let $x \in \mathbb{R}$. Observe that the two events $\{U \leq F(x)\}$ and $\{F^{\leftarrow}(U) \leq x\}$ imply each other. Hence, $P(F^{\leftarrow}(U) \leq x)=F(x)$. Consequently $E(X)=E(F^{\leftarrow}(U))$. The proof follows from the fact that $E(F^{\leftarrow}(U))$ is equal to the right-hand side of \eqref{1-eq-lemma-ast-expectation-inequality-L-1}.
\end{proof}

\begin{proof}[Proof of \cref{1-lemma-increasing-function-asymptotic-usual-stochastic-order}]
	Let $u \in (0,1)$. Note that $g^{\rightarrow}(y)<f^{\leftarrow}(u)$ implies that $f(g^{\rightarrow}(y))<u$. Again, if $g^{\rightarrow}(y) \geq f^{\leftarrow}(u)$, then $f(g^{\rightarrow}(y)) \geq f(f^{\leftarrow}(u)) \geq u$, where the first and the second inequalities are due to nondecreasingness and right-continuity of $f$, respectively. Hence, $f(g^{\rightarrow}(y)) \geq u$ if and only if $g^{\rightarrow}(y) \geq f^{\leftarrow}(u)$. Then $(f \circ g^{\rightarrow})^{\leftarrow}(u)=(g^{\rightarrow})^{\leftarrow}(f^{\leftarrow}(u))$. Observe that, for any function $h:\mathbb{R} \to \mathbb{R}$, $(h^{\rightarrow})^{\leftarrow}(x) \leq h(x)$, for every $x \in \mathbb{R}$. Furthermore, if $h$ is nondecreasing and left-continuous, then $(h^{\rightarrow})^{\leftarrow}(x)=h(x)$, for every $x \in \mathbb{R}$. Then we have $(g^{\rightarrow})^{\leftarrow}(f^{\leftarrow}(u))=g(f^{\leftarrow}(u))$. Hence, $(f \circ g^{\rightarrow})^{\leftarrow}(u)=g(f^{\leftarrow}(u))$.
\end{proof}

\begin{proof}[Proof of \cref{1-lemma-phi-F-inverse-phi-right-continuous}]
	Suppose that $F(x) \leq \phi^{\rightarrow}(u)$. If $\phi$ is left-continuous, then it follows that $\phi(F(x)) \leq u$. Thus, $(\phi \circ F)^{\rightarrow}(u) \geq F^{\rightarrow}(\phi^{\rightarrow}(u))$. If $\phi$ is right-continuous, then proceeding in a similar way, we get $(\phi \circ F)^{\leftarrow}(u) \leq F^{\leftarrow}(\phi^{\leftarrow}(u))$. The proof is completed by noting that the reverse inequalities $(\phi \circ F)^{\rightarrow}(u) \leq F^{\rightarrow}(\phi^{\rightarrow}(u))$ and $(\phi \circ F)^{\leftarrow}(u) \geq F^{\leftarrow}(\phi^{\leftarrow}(u))$ hold for every distortion function $\phi$ and every cdf $F$.
\end{proof}
	
\end{document}